\documentclass[12pt,a4paper,twoside,english]{smfart}

\usepackage{lmodern,mdframed,xcolor}
\usepackage[utf8]{inputenc}
\usepackage[francais]{babel}

\usepackage[OT2,T1]{fontenc}

\font\rurm=wncyr10 scaled \magstep1

\usepackage{amssymb}
\usepackage{amsmath}
\usepackage{smfthm,mathabx}
\usepackage{enumerate}
\usepackage{amsthm}
\usepackage{amsfonts}
\usepackage{graphicx}
\usepackage[colorlinks=true,linkcolor=black,citecolor=black,urlcolor=black]{hyperref}
\usepackage{fancyhdr}
\usepackage[all,cmtip]{xy}
\usepackage{alltt}
\usepackage{footmisc}
\usepackage{smfthm}

\usepackage{calrsfs}

\usepackage{fullpage}

\xyoption{all}

\def\Q{{\mathbb Q}}
\def\Z{{\mathbb Z}}
\def\fq{{\mathbb F}}

\def\b{{\mathfrak b}}
\def\p{{\mathfrak p}}
\def\P{{\mathfrak P}}

\def\q{{\mathfrak q}}

\def\d{{\partial}}

\def\F{{\rm F}}
\def\N{{\rm N}}
\def\R{{\rm R}}
\def\d{{\rm d}}
\def\G{{\rm G}}
\def\K{{\rm K}}
\def\L{{\rm L}}
\def\F{{\rm F}}

\def\MM{{\rm M}}
\def\M{{\rm M}}

\def\UU{{\rm U}}

\def\V{{\rm V}}
\def\I{{\rm I}}
\def\spec{{\rm Spec }}
\def\Gg{{\rm G}}

\def\df{{\rm Def}}

\def\Tr{{\rm Tr}}

\def\disc{{\rm disc}}

\def\X{{\rm X}}

\def\O{{\mathcal O}}

\def\U{{\mathcal U}}
\def\E{{\mathcal E}}

\def\VV{{\mathcal V}}
\def\ff{{\mathcal F}}

\def\H{{\rm H}}

\def\Gal{{\rm Gal}}

\def\cd{{\rm cd}}

\def\Cl{{\rm Cl}}

\def\ker{{\rm ker}}

\def\Res{{\rm Res}}

\def\sl{{\mathfrak s} {\mathfrak l}}
\def\1{{\bf 1}}

\def\Hhat{{\hat{H}}}

\def\FF#1#2{{\displaystyle{\left(\frac{#1}{#2}\right)}}}

\def\RCG#1#2{{\rm RCG}_{#1}(#2)}

\def\sha{{{\textnormal{\rurm{Sh}}}}}
\def\CyB{{{\textnormal{\rurm{B}}}}}

\def\Sha{{\sha}^2}

\newenvironment{Question}{\begin{enonce}{Question}}{\end{enonce}}

\newtheorem{Theorem}{Theorem}

\newtheorem*{Example}{Example}

\newtheorem*{Corollary}{Corollary}

\newtheorem*{Fact}{Fact}

\author{Farshid Hajir, Christian Maire, Ravi Ramakrishna}
\address{Department of Mathematics \& Statistics, University of Massachusetts, Amherst, MA 01003, USA}
 \address{FEMTO-ST Institute, Universit\'e Bourgogne Franche-Comt\'e, CNRS,  15B avenue des Montboucons, 25000 Besan\c con, FRANCE} 
\address{Department of Mathematics, Cornell University, Ithaca, USA}
\email{hajir@math.umass.edu, christian.maire@univ-fcomte.fr, ravi@math.cornell.edu}

\thanks{This work started when the second author held a visiting scholar position at Cornell University,   funded by the program "Mobilit\'e sortante" of the R\'egion Bourgogne Franche-Comt\'e, during the 2017-18   academic year. It continued during visits to the Harbin Institute of Technology and Cornell University.    CM  thanks the  Department of Mathematics at Cornell University and the  Institute for Advanced Study  in Mathematics of HIT for providing beautiful research atmospheres.
The second author  was partially supported by the ANR project FLAIR (ANR-17-CE40-0012) and  by the EIPHI Graduate School (ANR-17-EURE-0002). The third author was supported by Simons Collaboration grant \#524863. All three authors were supported by 
Mathematisches Forschungsinstitut Oberwolfach for a Research in Pairs visit in January, 2019 and by
ICERM for a Research in Pairs visit in January, 2020. }

\begin{abstract} Let $p$ be a prime.  We define the deficiency of a finitely-generated pro-$p$ group $\G$ to be $r(\G)-d(\G)$ where $d(\G)$ is the minimal number of generators of $\G$ and $r(\G)$ is its minimal number of relations. For a number field $\K$, let $\K_\emptyset$ be the maximal unramified $p$-extension of $\K$, with Galois group $\G_\emptyset = \Gal(\K_\emptyset/\K)$. In the 1960s, Shafarevich (and independently Koch) showed that the deficiency of $\G_\emptyset$ satisfies  $$0\leq \df(\G_\emptyset) \leq 
\dim (\O_\K^\times/(\O_\K^{\times })^p),$$
relating the deficiency of $\G_\emptyset$ to the $p$-rank of the unit group $\O_\K^\times$ of the ring of integers $\O_\K$ of $\K$. In this work, we further explore connections between relations of the group $\G_\emptyset$ and the units in the tower $\K_\emptyset/\K$, especially their Galois module structure.   In particular, under the assumption that  $\K$ does not contain a primitive $p$th root of unity, we give an exact formula for $\df(\G_\emptyset)$ in terms of the number of independent Minkowski units in the tower. 
The method also allows us to infer more information about the relations of $\G_\emptyset$, such as their depth in the Zassenhaus filtration, which in certain circumstances makes it easier to show that $\G_\emptyset$ is infinite. We illustrate how the techniques can be used to provide evidence for the expectation that the Shafarevich-Koch upper bound is ``almost always'' sharp.
\end{abstract}



\begin{document}

\date{\today}

\title{Deficiency of $p$-class tower groups and Minkowski units}

\maketitle


Let $p$ be a prime number, and let $\K$ be a number field.  For a finite set $S$ of places of $\K$, let $\K_S$ be the maximal $p$-extension of $\K$ unramified outside $S$ and $\G_S=\Gal(\K_S/\K)$, its Galois group.  
Note in particular that $\K_\emptyset$ is the maximal pro-$p$ extension of $\K$ unramified everywhere. We call the extension $\K_\emptyset/\K$ the $p$-class field tower of $\K$ and the group $\G_\emptyset$ its $p$-class tower group.
Let 
$$d(\G_\emptyset) := \dim H^1(\G_\emptyset,\Z/p) \text{ and }  r(\G_\emptyset):=\dim H^2(\G_\emptyset,\Z/p)$$ be, respectively, the minimal number of generators and relations of $\G_\emptyset$. By class field theory, the maximal abelian quotient of $\G_\emptyset$ is isomorphic to the $p$-Sylow subgroup of the class group of $\K$, and is therefore finite. It follows that $r(\G_\emptyset)\geq d(\G_\emptyset)$, prompting us to define 
 $\df(\G_\emptyset):=r(\G_\emptyset)-d(\G_\emptyset)$ as the {\it deficiency} of  $\G_\emptyset$.\footnote{We should note that in most of the group theory literature, the deficiency of $\G$ is defined as $d(\G)-r(\G)$.} By the Burnside Basis Theorem, $d(\G_\emptyset)$ is the $p$-rank of the class group of $\K$ and is in particular computable in any given case. By contrast, we do not know an algorithm for computing $r(\G_\emptyset)$. However, thanks to 
 the celebrated work of Shafarevich \cite{Shaf2} (and, independently, Koch -- see for example \cite[Chapter 11]{Koch}) we know that
 $$0\leq \df(\G_\emptyset) \leq d (\O_\K^\times),$$ 
 where $d(\O_\K^\times):=d(\O_\K^\times/(\O_\K^{\times})^p)$ is the $p$-rank of the unit group $\O_\K^\times$ of the ring of integers $\O_\K$ of $\K$. 
 We recall that if $\K$ has $r_1$ embeddings into $\mathbb R$ and $r_2$ pairs of complex conjugate embeddings into $\mathbb C$, then $d(\O_\K^\times) = r_1 + r_2 -1 + \delta$ where $\delta$ is $1$ or $0$ according to whether $\K$ contains a primitive $p$th root of unity $\zeta_p$ or not. The pioneering number-theoretic work of Shafarevich on the above deficiency bound \cite{Shaf2} subsequently led to the group-theoretic work of Golod and Shafarevich \cite{GS}. This pair of papers gave a criterion for the infinitude of $\G_\emptyset$ and produced the first such examples. More historical details can be found in \cite{Shaf1}, \cite{Roq}, \cite{Koch}, and \cite{NSW}.

 \medskip
 
Now let us recall the notion of a Minkowski unit. If $\F/\K$ is a Galois extension of number fields, an element $\varepsilon \in \O_\F^\times$ is called a Minkowski unit for $\F/\K$ if the subgroup of $\O_F^\times$ generated by $\varepsilon$ and all of its $\Gal(\F/\K)$ conjugates is ``maximal'' in some sense. 
For the precise notion of maximality that we will use, see below.
 
 \medskip
In this work, we relate the existence of Minkowski units along the tower $\K_\emptyset/\K$ to the deficiency of 
$\G_\emptyset$, as well as to the depth of its relations.
To do this,  we introduce a constant $\lambda$ measuring the free-part of the structure of the units in~$\K_\emptyset/\K$ (see  \S \ref{section:alongthetower}). Here 
free-part means the following: if $\F/\K$ is a finite Galois extension in $\K_\emptyset/\K$ with Galois group $\G$, we are interested in the $\fq_p[\G]$-structure of $\fq_p \otimes \O_\F^\times$, that is  the units $\O_\F^\times$ modulo $p$th powers. 
Recall that since $\G$ is a $p$-group, the category of $\fq_p[\G]$-modules
is not  semisimple. 
When the $\fq_p[\G]$-free part of  $\fq_p \otimes \O_\F^\times$  is nontrivial, we say the  extension $\F/\K$ admits a {\it Minkowski unit} 
(see \S \ref{section:minkowski} for further details).
It is not difficult to see that   the number of independent Minkowski units is non-increasing and stabilizes as we move up the tower
$\K_\emptyset/\K$ and therefore after a finite number of steps reaches a constant value we denote 
$\lambda:=\lambda_{\K_\emptyset/\K}$.  We also define~$\beta$ to be 
 $$\beta:=\left\{ \begin{array}{cl} d\left(\frac{\O_\K^\times \cap (\O_{\K_\emptyset}^{\times})^p  }{ ( \O_\K^{\times})^p }\right) & \zeta_p \in \K \\ 
 0 & \mbox{otherwise} \end{array}\right. .$$ 
\smallskip
Note that when $\zeta_p\in \K$, if we set $\L= \K_\emptyset \cap \K( \sqrt[p]{\O_\K^\times})$, then $[\L:\K]=p^\beta$. Thus,
 $0\leq \beta \leq \mbox{min}\left(r_1+r_2, d(\G_\emptyset)\right)$.
 Moreover, we note that $\beta >0$ if and only if $\K(u^{1/p})/\K$ is a Galois degree $p$ unramified extension for some $u\in \O_\K^\times$.
 \smallskip
 
 In the much better understood wild case, i.e. when $S$ contains all the primes of $\K$ above $p$ (as well as the infinite places when $p=2$), we can apply the very powerful global duality theorem for  $\G_S$ to obtain an exact 
 and easily computable formula for the deficiency of $\G_S$.  
Theorem \ref{theo:chi0} below (see Theorem \ref{theo:chi}) 
is a first 
step toward refining our understanding the unramified situation by relating it to the presence of Minkowski units in  $\K_\emptyset/ \K$.

\begin{Theorem} \label{theo:chi0} 
One has $$d(\O_\K^\times)- \lambda -\beta \leq \df(\G_\emptyset) \leq d(\O_\K^\times)- \lambda.$$
In particular if  $\O_\K^\times \cap (\O_{\K_\emptyset}^{\times})^p=(\O_\K^{\times})^p$ or if $\K$ does not contain a primitive pth root of unity, then 
$\df(\G_\emptyset)= d(\O_\K^\times)- \lambda$.
\end{Theorem}

We give two proofs of this result (see proof of Theorem \ref{theo:chi}). In our second proof, which is more constructive, we realize $\G_\emptyset$ as a quotient of some $\G_S$, where $S$ is a well-chosen finite set of prime ideals of $\O_\K$ coprime to $p$ (see also Notations at the end of this section), and we use the Hochschild-Serre spectral sequence induced by the natural map $\G_S \twoheadrightarrow \G_\emptyset$ 
to produce $d(\O_\K^\times)- \lambda -\beta$ distinct elements of $\sha^2 =H^2(\G_\emptyset,\Z/p)$.

\begin{rema} \label{Remark:makesha} As mentioned earlier, the non-negativity of $\df(\G_\emptyset)$ follows from a basic group-theoretical property of $\G_\emptyset$, namely that its maximal abelian quotient is finite. In other words, one knows that $\G_\emptyset$ has at least $d(\G_\emptyset)$ relations. For the group $\G_\emptyset$, one can concretely produce $d(\G_\emptyset)$ relations for $\G_\emptyset$ as follows. In Figure 1, we show a tower of fields $K \subset K' \subset L_1' \subset L_2'$ for whose definition, the reader may consult the Notations section at the end of this introduction. For the moment, the key point is that $L_2'/L_1'$ is an elementary abelian $p$-extension of dimension $d(\G_\emptyset)$. By the Chebotarev density theorem, we can choose $d(\G_\emptyset)$ primes of $\K$ whose Frobenii form a basis of the elementary $p$-abelian group $\Gal(\L'_2/\L'_1)$. Letting $S_1$ be the set consisting of these primes, in Section \ref{section:2.2}, we will describe in detail how applying  the Gras-Munnier Theorem (Theorem \ref{theo:Gras-Munnier})  and Lemma \ref{lemma1} (ii) to the primes in $S_1$ gives us $d(\G_\emptyset)$ distinct elements in $\sha^2 =H^2(\G_\emptyset,\Z/p)$, which in turn correspond to $d(\G_\emptyset)$ distinct relations in a minimal presentation of $\G_\emptyset$. We refer to relations constructed in this way as ``accessible'' via $S_1$.

 \begin{figure}[h]
$${
\xymatrix{ 
\label{Fig1}
& & \L'_2:=\K'\left(\sqrt[p]{\V_\emptyset}\right)    &  \\
& \L'_1:=\K'\left(\sqrt[p]{\O_\K^\times}\right) \ar@{-}[ur]& & \\
\K':=\K(\zeta_p) \ar@{-}[ur]&  & &\\
\K \ar@{-}[u] &  & &
}}$$ \caption{}\label{Fig1}
\end{figure}

A key observation we make in this work is that aside from the $d(\G_\emptyset)$ relations accessible via $S_1$, we can construct  $d(\O_\K^\times)- \lambda -\beta$ additional relations via a modification of this construction using a further set $S_2$ of  auxiliary primes whose Frobenii span a Galois group  in a more complicated tower of  governing fields (Figure \ref{Fig2}) described in \S \ref{section:2.2}. The existence of such primes is tied up with the Galois module structure of units in the Hilbert $p$-class field tower. We refer to the resulting relations as ``extra'' relations ``detected'' by $S_2$. This set of ideas leads to the lower bound for  $\df(\G_\emptyset)$ in the theorem. The upper bound, on the other hand, is a consequence of a result of Wingberg \cite{Wingberg-FM}.  When $\beta=0$ (which is always the case if $K$ does not contain a primitive $p$th root of unity), these upper and lower bounds coincide, in which case {\em all} the relations are either accessible by $S_1$ or detected by $S_2$. But when $\beta > 0$, only $d(\O_\K^\times)- \lambda -\beta$ of the relations are constructible in this way, and there is an open question as to whether there exist (up to) $\beta$ additional ``elusive'' relations.
\end{rema}

\smallskip

Thanks to the work of Labute \cite{Labute}, Schmidt \cite{Schmidt} and others we know that there are special sets $S$ (finite and tame) for which $\G_S$ is of cohomological dimension~$2$; however, their methods do not allow $S$ to be empty.  In particular, the question of the computation of the cohomological dimension of  $\G_\emptyset$ 
 has only been resolved in a few cases, namely  when $\G_\emptyset$ is known to be finite. 
A consequence of Theorem \ref{theo:chi0} is the following (Theorem \ref{theo:coho_dimension}):

\begin{Corollary} \label{theo:r-d_coh_dim} 
  Let $\K$ be a number field such that 
\begin{itemize}
\item[(i)] $\K$ contains a primitive $p$th root of unity;
\item[(ii)]  $ \O_\K^\times\cap (\O_{\K_\emptyset}^{\times})^p= (\O_\K^{\times})^p$.
\end{itemize}
Then $\dim H^3(\G_\emptyset,\fq_p)>0$. 
Moreover:
\begin{itemize}
 \item[$-$] If $\dim H^3(\G_\emptyset,\fq_p)=1$, then $\G_\emptyset$ is finite or  of cohomological dimension $3$;
 \item[$-$] If $\df(\G_\emptyset)=0$ and  $\G_\emptyset$ is of cohomological dimension  $3$, then $\G_\emptyset$ is a Poincar\'e duality group.
\end{itemize}
\end{Corollary}

\medskip
We also deduce (Corollary \ref{theo:r-d}):

\begin{Corollary}  Let $\K$ be a number field such that 
\begin{itemize}
\item[(i)] $\K$ contains a primitive $p$th root of unity;
\item[(ii)]  $ \O_\K^\times\cap (\O_{\K_\emptyset}^{\times})^p= (\O_\K^{\times })^p$;
\item[(iii)]  $\df(\G_\emptyset)=0$.
\end{itemize}
Then for every open normal subgroup $\H$ of $\G_\emptyset$, one has  $\df(\H)=0$.
\end{Corollary}

Note that in the situation of the Corollary, i.e. when $\df(\H)=0$ for all open subgroups of $\G$, $\lambda$ is maximal all along the tower.
When $\df(\G_\emptyset)=0$ and the tower is finite, 
$\G_\emptyset$ is either cyclic or, when $p=2$, a generalized quaternion group. 
We know of no examples with $\df(\G_\emptyset)=0$ and  $\#\G_\emptyset=\infty$.
Observe also that Poincar\'e groups of dimension $3$ have deficiency zero.

When $\G_\emptyset$ is infinite we suspect   $\df(\G_\emptyset)$ to be maximal (namely equal to $d(\O_\K^\times)$) very often in accordance with the heuristics of Liu-Wood \cite{Liu-Wood}.
In fact, 
we elaborate a strategy to investigate maximality of the deficiency by testing for the presence of Minkowski units through computer computation. We further note that if in the first steps of the tower  $\K_\emptyset/\K$ there is  some Minkowski unit, preventing us from concluding that $\df(\G_\emptyset)$ is maximal,  it implies that the group $\G_\emptyset$ can be described by relations of high depth in the Zassenhaus filtration. Denote by $(\K_n)$ the sequence in $\K_\emptyset/\K$ where $\K_1:=\K$ and $\K_{n+1}$ is the maximal elementary abelian $p$-extension of $\K_n$ in $\K_\emptyset/\K$. Put $\H_n=\Gal(\K_n/\K)$. 
Let $r_{\textrm{max}}=d(\G_\emptyset)+d(\O_\K^\times)$ be the maximal possible value of $r(\G_\emptyset)$. To each presentation of a pro-$p$ group, there is associated a Golod-Shafarevich polynomial; for the basic facts of these polynomials, see \S \ref{sec:GS}. Golod and Shafarevich proved that if this polynomial vanishes on the open unit interval, then the group must be infinite. 
 In \S 4, we prove  the following result (see Theorem \ref{main-theorem2}). 
\medskip
 
\begin{Theorem}  Let $\lambda_n$ be the number of independent $\fq_p[\H_n]$-Minkowski units in $\K_n$.
 Then   $\G_\emptyset$  can be generated by $d(\G_\emptyset)$ generators and  $r_{\textrm{max}}$ relations $\{\rho_1,\cdots, \rho_{r_{\textrm{max}}}\}$ such that  at least $\lambda_n$  relations  are of depth greater than $2^{n}$.
 Hence, we can take $1-d(\G_\emptyset)t+(r_{\textrm{max}}-\lambda_n)t^2+\lambda_nt^{2n}$ as a Golod-Shafarevich polynomial for $\G_\emptyset$.
\end{Theorem}

The more familiar Golod-Shafarevich polynomial in this context is $1-d(\G_\emptyset)t+r_{\textrm{max}}t^2$, which is {\it less} likely to have a root and thus indicate $\#\G_\emptyset =\infty$. Also, we will allow the possibility of $n=\infty$, that is there may be fewer than $r_{\textrm{max}}$ relations. 

\

 The imaginary quadratic case is particularly easy to study (here $p=2$). 
Indeed when~$\K$ is an imaginary quadratic field, one has $\df(\G_\emptyset) \in \{0,1\}$. Here we  show that, almost always, there is no Minkowski unit in any quadratic extension $\F/\K$ of $\K_\emptyset/\K$, which implies $\df(\G_\emptyset)=1$.
Denote by $\ff$ the set of imaginary quadratic fields, and for $X\geq 2$, put $$\ff(X)=\{ \K \in \ff, \ |\disc(\K)|\leq X\}, \ \  \ff_0(X)=\{\K \in \ff(X), \ \df(\G_\emptyset)=0\}.$$ 
We obtain that almost all the time $\df(\G_\emptyset)=1$ (see Theorem \ref{theo:density_quadratic}):

\begin{Theorem} \label{TheoremC} Let $\K$ be imaginary quadratic and $p=2$. One has 
$$\frac{\#\ff_0(X)}{\#\ff(X)} \leq C\frac{\log \log X}{\sqrt{\log X}},$$ where $C$ is an absolute constant.
In particular, the proportion of imaginary quadratic fields of discriminant at most $X$ for which $\df(\G_\emptyset)=0$ tends to zero as $X \rightarrow \infty$. 
\end{Theorem}

\begin{Example} Take $p=2$. Our method allows us to show that for $\K=\Q(\sqrt{-5460})$, the example studied extensively by Boston-Wang \cite{Boston-Wang}, one has 
$\df(\G_\emptyset) =r-d= 5-4=1.$
\end{Example}

\

{\bf Notations.}

$\bullet$ Let $p$ be a prime number, and let $\K$ be a number field.

$\bullet$ We denote by 
\begin{itemize}
\item[$-$] $\O_\K$ the ring of integers of $\K$, and by $\O_\K^\times$ the group of units of $\O_\K$,
\item[$-$]   $\E_\K=\fq_p \otimes \O_\K^\times$, the units modulo the $p$th-powers,
\item[$-$]  $\K^H$ the Hilbert $p$-class field of $\K$,
\item[$-$]  $\Cl_\K$ the $p$-Sylow subgroup of class group of $\K$.
\end{itemize}
 
 $\bullet$ Let $\zeta_p \in \Q^{alg}$ be a primitive $p$th root of $1$. Put $\delta:=\delta_{\K,p}:=1$  when $\zeta_p \in \K$, $0$ otherwise.
 
$\bullet$  Let $S=\{\p_1,\cdots,\p_s\}$ be a finite set of prime ideals  of $\K$. We identify a prime $\p \in S$ with the  place  $v$ it defines. 
\begin{itemize}
\item[$-$] We assume each  $\p_i$ is tame (prime to $p$) and satisfies   $|\O_\K/\p_i| \equiv 1 ({\rm mod} \ p)$. 
\item[$-$] We denote by  $\RCG{\K}{\p_1,\cdots, \p_s}$ the $p$-Sylow subgroup of the ray class group of $\K$ of modulus $\p_1\cdots \p_s$.
When $S=\emptyset$, one has $ \RCG{\K}{\emptyset}=\Cl_\K$.
\item[$-$] 
Let  $\K_S$  be the maximal pro-$p$ extension of $\K$ unramified outside $S$, put  $\G_S=\G_{\K,S}=\Gal(\K_S/\K)$.
\item[$-$] By class field theory, one has $\G_S^{ab} \simeq \RCG{\K}{\p_1,\cdots, \p_s}$.
\item[$-$]
Put
$\V_S:=\{ x\in \K^\times, \  (x)=I^p {\rm \ as \ a \ fractional \ ideal \ of\ } \K; \ x\in (\K^\times_v)^p, \forall v \in S\}$. Then $\V_S \supset (K^{\times})^p$ and we have the exact  sequence:
$$0 \to \O_\K^{\times}/ \O_\K^{\times p} \to \V_\emptyset/(\K^{\times})^p \to \Cl_{\K}[p] \to 0.$$
\end{itemize}
  
$\bullet$ 
If $\M$ is a $\Z$-module, we set $d(\M)= \dim_{\fq_p} (\fq_p\otimes \M)$.
\begin{itemize}
 \item[$-$] 
When $\G$ is a pro-$p$ group, we denote $d(\G)=d(\G^{ab})$, where $\G^{ab}=\G/[\G,\G]$.
\item[$-$] 
If $(r_1,r_2)$ is the signature of $\K$ and $\delta$ equals $1$ or $0$ as $\K$ contains or does not contain the $p$th roots of unity. By Dirichlet's Theorem  $d(\O_\K^\times)=r_1+r_2-1+\delta$.
\item[$-$] From the exact sequence above, $d(\V_\emptyset / (\K^{\times})^p)=d(\O_\K^\times)+d(\Cl_\K)$.
\end{itemize}

$\bullet$ Unless otherwise specified, all  cohomology groups have  $\Z/p\Z$ as their coefficient group.
\begin{itemize}
 \item[$-$]  Hence $d(\G_\emptyset):=d(H^1(\G_\emptyset))= \dim H^1(\G_\emptyset,\Z/p\Z)$  and
  $r(\G_\emptyset):=\dim H^2(\G_\emptyset,\Z/p\Z)$.
 \item[$-$] The deficiency\footnote{Usually the deficiency of a pro-$p$ group $\G$ is the quantity $d(\G)-r(\G)$, but for clarity we use the opposite} $\df(\G_\emptyset)$ of $\G_\emptyset$ is defined to be $r(\G_\emptyset)-d(\G_\emptyset)$.

\end{itemize}

\medskip

For the computations in this paper we have used the programs GP-PARI \cite{Pari} and Magma \cite{magma} and have assumed the GRH to speed up the computations.





\

\section{Preliminaries}
In this section we develop the results we need to detect elements of $H^2(\G_\emptyset) =\sha^2$ as described in the Remark in the Introduction.
In particular, Lemma \ref{lemma1} shows how one can detect elements of $\sha^2$ via ramified extensions of 
$\K_\emptyset$; 
 we illustrate our strategy by finding $\sha^2$ for  the field $\Q(\sqrt{-5460})$ with $p=2$.

 In \S \ref{section:alternative_proof},  we relate $\df(\G_\emptyset)$ to norms of units from number fields in the tower $\K_\emptyset/\K$. In \S \ref{section:minkowski} we develop the basics of the theory of Minkowski units and show, using the Gras-Munnier Theorem \ref{theo:Gras-Munnier}, that the existence of a Minkowski unit in some number field $\F$ in the tower $\K_\emptyset/\K$ follows when $\G^{ab}_{\F,\{\p\}}=\G^{ab}_{\F,\emptyset}$ for some prime $\p$ of $\K$.
\subsection{Saturated sets,  and a spectral sequence}

\subsubsection{Degree-$p$ cyclic extension with prescribed ramification}
Take $p$, $\K$ and $S$ as in \S ``Notations''.
The fields    of Figure \ref{Fig1} are called {\it governing fields} as the existence of a $\Z/p$-extension of $\K$ ramified exactly at a given set of primes depends on  their Frobenii in these extensions. See Theorem \ref{theo:Gras-Munnier} below.

For each prime ideal $\p\in S$, let us choose a prime ideal $\P|\p$ of $\O_{\L'_2}$, and  denote by $\sigma_\p:=  \displaystyle{ \FF{\L'_2/\K'}{\P}}$, the Frobenius at $\P$ in the governing extension $\L'_2/\K'$.

 Using that ${\L'_2}$ is formed by taking $p$th roots of elements of $\K$ (not $\K'$), one can show that $\sigma_\p$ depends, up to a nonzero scalar multiple, only on $\p$. This serves our purposes. 
By abuse we also denote by $\sigma_\p$ its restriction to $\L_1'$.
One says that the Frobenii $\sigma_\p$, $\p \in S$, satisfy a {\it  nontrivial relation}  if 
$$\prod_{\p \in S} \sigma_\p^{a_\p}=1,$$
in $\Gal(\L'_2/\K')$ (or  in $\Gal(\L'_1/\K')$) with the $a_i \in \Z/p\Z$  not all zero. 
Thus the existence of a nontrivial relation is independent of the ambiguity in the choice of $\sigma_\p$.

\begin{theo}[Gras-Munnier \cite{Gras-Munnier}] \label{theo:Gras-Munnier} 
Let $S=\{\p_1,\cdots, \p_t\}$ be a set of tame prime ideals of~$\K$. One has:
\begin{enumerate}
\item[$(i)$]  $d(\G_S) \neq d(\G_\emptyset)$, if and only if the $\sigma_\p$, $\p \in S$, satisfy a nontrivial relation in $\Gal(\L'_2/\K')$.
\item[$(ii)$] 
$|\G_S^{ab}| > |\G_\emptyset^{ab}|$ if and only if  the $\sigma_\p$, $\p \in S$, satisfy a nontrivial relation in $\Gal(\L_1'/\K')$.
\end{enumerate}
\end{theo}

For a generalization of Theorem \ref{theo:Gras-Munnier}, see \cite[Chapter V]{gras}.

\begin{rema} \label{remark:Kummer_radical}
Let us observe that the Kummer radical of $\L_1'/\K'$ is $\O_\K^\times (\K'^\times)^p/(\K'^\times)^p$ which is isomorphic to  $ \E_\K$ since $\O_\K^\times \cap (\K'^\times)^p=(\O_\K^\times)^p$ and $[\K':\K]$ is coprime to $p$.
For the same reason, $\V_\emptyset/(\K^\times)^p$ is the Kummer radical of $\L'_2/\K'$.
\end{rema}

\subsubsection{Saturated sets}
For $v \in S$, we denote by $\G_v$ the absolute Galois group of the maximal pro-$p$ extension $\overline{\K}_v$ of the completion $\K_v$ of $\K$ at $v$.
Let $\Sha_S$ be the kernel of the localization map of $H^2(\G_S,\fq_p)$: $$\Sha_S= \ker\big(H^2(\Gg_S,\fq_p)\rightarrow \oplus_{v\in S} H^2(\Gg_v,\fq_p)\big). $$ Put $\CyB_S = \big(\V_S/(\K^\times)^p\big)^\vee$; then
one has  (see Theorem 11.3 of \cite{Koch}) $ \Sha_S \hookrightarrow \CyB_S$.
 When $S$ contains the places of $\K$ above $\{p,\infty\}$ this map is an isomorphism and $\Sha_S$ is dual to 
$$\sha^1_S(\mu_p):=\ker\big(H^1(\Gg_S,\mu_p)\rightarrow \oplus_{v\in S} H^1(\Gg_v,\mu_p)\big).$$
The failures of the isomorphism and duality in the tame case are reasons it is especially challenging.

\begin{defi} 
The $S$  set of places $\K$ is called {\it saturated} if $\V_S/(\K^\times)^p=\{1\}$. 
\end{defi}

As consequence of Theorem \ref{theo:Gras-Munnier}, one has (see \cite[Theorem 1.12]{HMR_sha})

\begin{theo} \label{theo_saturated}
A finite tame set  $S$  is saturated if and only if, the Frobenius $\sigma_\p$, $\p \in S$, span the elementary $p$-abelian group $\Gal(\K'(\sqrt[p]{\V_\emptyset})/\K')$.
\end{theo}

We recall below the formula of Shafarevich  applied in the case where $S$ is tame  (see for example \cite[Chapter X, \S 7, Corollary 10.7.7]{NSW}): 
\begin{eqnarray} \label{rank_formula}
d(\G_S)=|S|-d(\O_\K^\times) +d(\V_S/(\K^\times)^p).
\end{eqnarray}
Hence when $S$ is saturated, one has $d(\G_S)=|S|-d(\O_\K^\times)$.

\subsubsection{Spectral sequence} \label{section:spectral_sequence}

Let us start with the natural exact sequence
$$1 \longrightarrow \H_S \longrightarrow \G_S \longrightarrow  \G_\emptyset \longrightarrow 1,$$
where the group $\H_S$  is the closed normal subgroup of $\G_S$ generated by the inertia $\tau_\p \in \G_S$,  $\p \in S$.
Set
$$\X_S:=\H_S/[\H_S,\G_\emptyset]\H_S^p.$$

Recall as $\G_\emptyset$ is a pro-$p$ group, the compact ring  $\fq_p\ldbrack \Gg_\emptyset \rdbrack$ is local and acts continuously on $\H_S/[\H_S, \H_S]\H_S^p$. We give an easy lemma that can be found in~\cite{HMR_sha} (see Lemmas 1.11 and 1.12).

\begin{lemm} \label{exact-sequence1} \label{lemma1}
Let $ S$ be a finite set of tame prime ideals of $\O_\K$.
\begin{enumerate}
\item[$(i)$] The $\fq_p\ldbrack \Gg_\emptyset \rdbrack$-module $\H_S/[\H_S, \H_S]\H_S^p$  is  topologically  finitely generated by at most  $|S|$ elements.
\item[$(ii)$]  One has  the exact sequence
$$ 
 1 \longrightarrow H^1(\Gg_\emptyset) \longrightarrow H^1(\Gg_{S}) \longrightarrow \X_S^\vee \longrightarrow \Sha_\emptyset \longrightarrow \Sha_S.$$
 In particular, if  $S$ is such that $H^1(\Gg_\emptyset) \simeq H^1(\Gg_{S})$, then $\X_S^\vee \hookrightarrow \Sha_\emptyset$. If moreover  $ S$ is saturated then
 $\X_S^\vee \simeq \Sha_\emptyset$.
\end{enumerate}
\end{lemm}

To conclude this subsection, let us  observe the following:
Let $\F_0/\K_\emptyset$ be a  cyclic extension of degree $p$ in $\K_S/\K$ such that $\F_0/\K$ is Galois.
Then $\F_0$ comes from a finite level: there exists a finite extension $\F/\K$ and a cyclic extension $\F_1/\F$ of degree $p$, ramified at some places above $S$,  such that $\F_2=\K_\emptyset\F_1$.
The estimate for $\dim H^2(\Gg_\emptyset)$ can be done by using the previous lemma, typically  by seeking the  fields~$\F_0$: this is the spirit of the method  involving the Hochschild-Serre spectral sequence.

\begin{exem}[The field $\Q(\sqrt{-5460}$)] \label{field:5460}
Set $p=2$ and $\K=\Q(\sqrt{-5460})$. 
The rational primes in $\{43,53,101,149,157\}$ all split in $\K$. Let $S=\{\p_{43},\p_{53},\p_{101},\p_{149},\p_{157}\}$, the first primes above each of these as Magma computes them. We denote the abelian group $\prod^d_{i=1} \Z/a_i$ by  $(a_1,\cdots,a_d)$.
Computations show that:
\begin{enumerate}
\item[$(i)$] $\RCG{\K}{\emptyset}=(2,2,2,2)$;
\item[$(ii)$] $\RCG{\K}{\p_{53},\p_{101},\p_{149},\p_{157}}=(4,8,8,8)$;  Furthermore, for each of these primes $\p$ we compute  $\RCG{\K}{\p} =(2,2,2,4)$. 
\item[$(iii)$] $\RCG{\K}{S}= (8,8,8,8)$.
\end{enumerate}
Since $\O_\K^\times = \pm 1$, one easily sees $d( \V_\emptyset/ (\K^\times)^2) = 4+1=5$ hence
$S$ is saturated by $(iii)$ and equality (\ref{rank_formula}), and $(ii)$ implies  $d(\X_{\{\p_{53},\p_{101},\p_{149},\p_{157}\}})= 4$ and then $\d(\X_S )\geq 4$. 
As $\X_{\{\p\}}$ is nontrivial  for $\p \in \{\p_{53},\p_{101},\p_{149},\p_{157}\}$, we see  there is a 
 quadratic extension above $\K_\emptyset$ ramified at $\p$. We have produced four independent elements of 
 $\sha^2_\emptyset$.
 
\smallskip

Now take $\F=\K(i) \subset \K_\emptyset$; $\p_{43}$ is inert in $\F/\K$.  An easy computation shows that $\RCG{\F}{\emptyset}=(4,2,2,2)$ and $\RCG{\F}{\p_{43}}=(4,4,2,2)$. 
As $\F_\emptyset=\K_\emptyset$ we have an extension over $\K_\emptyset$ ramified only at $\p_{43}$, so
$d(\X_S) = 5$, and by Lemma \ref{exact-sequence1} we conclude that $r(\G_\emptyset)=5$.
\end{exem}


\subsection{Universal norms and relations} 
\label{section:alternative_proof}

Put $\displaystyle{\O_{\K_\emptyset}^\times=\bigcup_{\F} \O_\F^\times}$, where $\F/\K$ run through the finite Galois extensions in $\K_\emptyset /\K$.
Recall the following theorem due to Wingberg \cite{Wingberg-FM}; see also \cite[Theorem 8.8.1, Chapter VIII, \S 8]{NSW} 
 where we take $S=T=\emptyset$, $\mathfrak c$ to be the full class of finite $p$-groups and $A=\mathbb Z$. 
We have written the results there in our notation.
 
\begin{theo}[Wingberg] \label{theo:wingberg} 
One has $\Hhat^i(\G_\emptyset,\O_{\K_\emptyset}^\times)\simeq \Hhat^{3-i}(\G_\emptyset,\Z)^\vee$.
\end{theo}

\medskip
The exact sequence $0 \longrightarrow \Z \stackrel{\times p}{\longrightarrow} \Z
\longrightarrow \Z/p\Z \longrightarrow 0$ gives:
$$0 \longrightarrow
H^2(\G_\emptyset,\Z)/p \longrightarrow H^2(\G_\emptyset,\Z/p\Z)
\longrightarrow H^3(\G_\emptyset,\Z)[p] \longrightarrow 0.$$

Taking the Pontryagin dual, one obtains:
\begin{eqnarray} \label{suite_exacte1} 
0 \longrightarrow H^3(\G_\emptyset,\Z)^\vee/p \longrightarrow
H^2(\G_\emptyset,\Z/p\Z)^\vee \longrightarrow
H^2(\G_\emptyset,\Z)^\vee[p] \longrightarrow 0.
\end{eqnarray} 
By Theorem \ref{theo:wingberg}:
$$H^2(\G_\emptyset,\Z)^\vee\simeq  H^1(\G_\emptyset,
\O_{\K_\emptyset}^\times),  \,\, and \, \, H^3(\G_\emptyset,\Z)^\vee
\simeq \hat{H}^0(\G_\emptyset,\O_{\K_\emptyset}^\times) .$$
Recall $$\hat{H}^0(\G_\emptyset,\O_{\K_\emptyset}^\times) \simeq
\lim_{\stackrel{\leftarrow}{\F}} \O_\K^\times/\N_{\F/\K}\O_\F^\times,$$
where $\F/\K$ run through the finite Galois extensions in
$\K_\emptyset /\K$, $\N_{\F/\K}$ is the norm in $\F/\K$, and
$H^1(\G_\emptyset, \O_{\K_\emptyset}^\times)$ is the $p$-part of  $\Cl_\K$ (see for
example \cite[Lemma 8.8.4, Chapter VIII, \S 8]{NSW}).

This observation associated to Theorem \ref{theo:wingberg} allows us
to prove:

\begin{coro} \label{prop:r-d_wingberg}
One has
$\displaystyle{\df(\G_\emptyset)=d(\E_\K/\N_{{\K_\emptyset}/\K}\E_{\K_\emptyset})}$,
where $\displaystyle{\N_{{\K_\emptyset}/\K}:=
\bigcap_{\F/\K}\N_{\F/\K}\E_\F}$.
{}In particular when $[\F:\K]$ is sufficiently large one has 
$\displaystyle{\df(\G_\emptyset)=d(\E_\K/\N_{{\F}/\K}\E_{\F})}$. {}
\end{coro}

\begin{proof} If  $\F/\K$ is a finite Galois extension in
$\K_\emptyset/\K$ then $(\O_\K^{\times})^{ [\F:\K]} \subset
\N_{\F/\K}\O_F^\times$, hence $\O_\K^\times /\N_{\F/\K}\O_F^\times$ is a
finite  abelian $p$-group and 
$\displaystyle{\lim_{\stackrel{\leftarrow}{\F}}
\O_\K^\times/\N_{\F/\K}\O_\F^\times}$ is an abelian pro-$p$ group
(obviously finitely generated).
Then $\displaystyle{\lim_{\stackrel{\leftarrow}{\F}}
\O_\K^\times/\N_{\F/\K}\O_\F^\times \simeq
\lim_{\stackrel{\leftarrow}{\F}} \Z_p \otimes
\left(\O_\K^\times/\N_{\F/\K}\O_\F^\times\right)}$.
But as $\Z_p$ is $\Z$-flat, one gets $\displaystyle{\Z_p \otimes
\left(\O_\K^\times/\N_{\F/\K}\O_\F^\times\right) \simeq E_\K /
\left(\Z_p\otimes \N_{\F/\K}\O_\F^\times\right) = E_\K /
\overline{\N_{\F/\K} \O_\F^\times}}$,
where $E_\K= \Z_p\otimes\O_\K^\times$, and where $\overline{\N_{\F/\K}
\O_\F^\times}$ is the closure of $\N_{\F/\K} \O_\F^\times$ in
$\Z_p\otimes \O_\K^\times$.
Hence, $$\lim_{\stackrel{\leftarrow}{\F}}
\O_\K^\times/\N_{\F/\K}\O_\F^\times \simeq E_\K /\bigcap_\F
\overline{\N_{\F/\K} \O_\F^\times}.$$
Thus $$\fq_p\otimes \lim_{\stackrel{\leftarrow}{\F}}
\O_\K^\times/\N_{\F/\K}\O_\F^\times \simeq E_\K/E_\K^p \bigcap_\F
\overline{\N_{\F/\K} \O_\F^\times} \simeq \E_\K/\N_{{\K_\emptyset}/\K}\E_{\K_\emptyset}.$$
The exact sequence (\ref{suite_exacte1}) becomes
\begin{eqnarray} \label{se1} 0 \longrightarrow
\E_K/\N_{{\K_\emptyset}/\K}\E_{\K_\emptyset} \longrightarrow H^2(\G_\emptyset,\Z/p\Z)^\vee \longrightarrow \Cl_\K[p] \longrightarrow 0,
\end{eqnarray}
and computing dimensions gives the result.
\end{proof}

For $\#\G_\emptyset < \infty$ it has been known for a long time that the number of relations of $\G_\emptyset$ is related to the norm of the units in the tower. See for example \S 2 of \cite{Roq}

\smallskip

As a consequence, one also has

\begin{coro} \label{prop:minoration_relations}
Let $\F/\K$ be  a finite Galois extension in $\K_\emptyset/\K$. Then $\df(\G_\emptyset) \geq d\left(\O_\K^\times/\N_{\F/\K}\O_\F^\times\right)$, and one has equality when $\F$ is sufficiently large.
\end{coro}

\begin{proof} Obvious by using  
$\E_\K/\N_{{\K_\emptyset}/\K} \E_{\K_\emptyset}  \twoheadrightarrow  \E_\K/\N_{\F/\K}\E_\F$. For the equality, use the fact that $\E_\K$ is finite.
\end{proof}

When $p=2$, if $-1$ is not a norm of a unit in a quadratic subextension $\F/\K$ of $\K_\emptyset/\K$, then $-1 \notin \N_{\K_\emptyset/\K}\O_{\K_\emptyset}^\times$, which implies $\df(\G_\emptyset) \geq 1$.
We will see that this condition appears almost all the time when $\K$ is an imaginary quadratic extension. We close this subsection with a basic fact.

\begin{Fact} \label{fact:DefgeqZero}
For $S$ a finite set of tame places, $ 0 \leq \df(\G_S)$.
\end{Fact}
\begin{proof}

  We refer to \cite{RZ}, especially Lemma $6.8.6$,
  for the facts we need concerning the homology of profinite groups.
From the exact sequence of compact groups 
$$0\longrightarrow \Z_p\stackrel{\times p}{\longrightarrow} \Z_p {\longrightarrow} \Z/p  \longrightarrow 0$$ we obtain the homology sequence
 $$\cdots \longrightarrow H_2(\G_S,\Z/p) \twoheadrightarrow H_1(\G_S,\Z_p)[p].$$
 As $H_1(\G_S,\Z_p)\simeq \G_S^{ab}$, 
 we have
  $d(\G_S)=d\big(\G_S^{ab}[p]\big)\leq d\big(H_2(\G_S,\Z/p)\big)=r(\G_S)$.
\end{proof}




\subsection{Minkowski units} \label{section:minkowski}

\subsubsection{} \label{subsection-algebraic_tools}
Recall that for  a finite group $\G$, the ring $\fq_p[\G]$ is a Frobenius algebra (see for example \cite[\S 62]{CR}): every free submodule of  an $\fq_p[\G]$-module $\MM$ is in direct sum so we may write $\MM=\fq_p[\G]^t \oplus \N$, where $\N$ is  torsion (for every element $n\in \N$, there exists $0\neq h \in \fq_p[\G]$ such that $h \cdot n =0$), and $t$ is uniquely determined (by Krull-Schmidt Theorem). 
Observe that if $\MM^\wedge$ is the Pontryagin dual of $\M$, then $t_{\G}(\MM)=t_{\G}(\MM^\wedge)$.

\begin{defi} If $\MM$ is a finitely generated $\fq_p[\G]$ module, denote by  $t:=t_{\G}(\MM)$, the $\fq_p[\G]$-rank of the maximal free submodule of $\MM$.
\end{defi}

We record some useful properties.
Let $\H \subset \G $ be a  subgroup of $\G$. 

$(i)$ Recall first that by Mackey's decomposition theorem, one has 
the isomorphism of $\fq_p[\H]$-modules $\Res_\H \fq_p[\G] \simeq \fq_p[\H]^{\oplus^{[\G:\H]}}$.

$(ii)$ Suppose moreover $\H \lhd \G$, and   denote by $N_\H=\sum_{h\in \H}h  \in \fq_p[\G]$ the norm map from $\H$. 
For an $\fq_p[\G]$-module $\M$ let $\M^H$ denote the invariants. 
Then one has easily  the isomorphism of $\fq_p[\G]$-modules

\begin{eqnarray} \label{fpgmodule1} 
\fq_p[\G/\H] \simeq \fq_p\otimes_{\fq_p[\H]} \fq_p[\G] \simeq \fq_p[\G]^\H
\end{eqnarray}
 and $N_\H(\fq_p[\G])=\fq_p[\G]^\H$ so 
\begin{eqnarray} \label{fpgmodule2} 
N_\H(\fq_p[\G])\simeq \fq_p[\G/\H]
\end{eqnarray}
as $\fq_p[\G/\H]$-modules.

\subsubsection{} 
Let $\F/\K$ be a finite Galois extension of number fields with Galois group~$\G$. 

\begin{defi} Let $\E_\F:=\fq_p\otimes \O_\F^\times$.
 We say that $\F/\K$ has a Minkowski unit (at $p$), if $\E_\F $  contains a {\it nontrivial free} $\fq_p[\G]$-submodule.  In other word, $\F/\K$ has a Minkowski unit if $t_\G(\E_\F) \geq 1$.
\end{defi}
Hence the quantity $t_\G(\E_\F)$ measures "the number" of independent Minkowski units in $\F/\K$.
 
 \medskip
 
If $(p,|\G|)=1$ then $\E_\F$   is a semisimple  $\fq_p[G]$-module. Determining the existence of Minkowski units is  more difficult when $(p, |\G|)=p$. 
Indeed, when  $\G$ is a $p$-group, and $\F/\K$ is unramified, the presence of  a Minkowski unit in $\F/\K$ is  probably a very strong condition as noted by Ozaki \cite[Lemma 2]{Ozaki} (for general $\G_S$, see \cite{Hajir-Maire-analytic}).

\begin{theo}[Ozaki] \label{theo:ozaki}
Let $\F/\K$ be a finite Galois extension in $\K_\emptyset/\K$.
Then  $$t_\G(\E_\F) \geq r_1+r_2 -C,$$
where $C\geq 0$ is a constant depending on $\Gal(\F/\K)$ and on $\Cl_\K$. Moreover $C\rightarrow \infty$ with $|\Gal(\F/\K)|$.
\end{theo}
Here as usual, $(r_1,r_2)$ is the signature of $\K$.

\subsubsection{Example} \label{section:explicit_example}

We want to illustrate the notion of Minkowski units. 

\begin{lemm} \label{lemma:minkowski-unit-1}
Let $\F/\K$ be a $p$-extension of Galois group $\G$. Let $S=\{\p_1,\cdots, \p_k\}$ be a set of tame primes of $\K$ that split completely in $\F/\K$.
If   $d(\G_{\F,S})=d(\G_{\F,\emptyset})$ then $t_{\G}(\V_{\F,\emptyset})\geq k$, and if  $|\G_{\F,S}^{ab}|=| \G_{\F,\emptyset}^{ab}|$ then $t_\G(\E_\F) \geq k$.
\end{lemm}

\begin{proof}  Observe first that $\G$ acts on $\Gal(\F'(\sqrt[p]{\V_{\F,\emptyset}})/\F')$ (resp. on $\Gal(\F'(\sqrt[p]{\O_\F^\times})/\F')$). Then by Remark  \ref{remark:Kummer_radical}, one has 
\begin{eqnarray} \label{equation00}
t_\G\big(\V_\F/(\F^\times)^p\big)
=t_\G\big(\Gal(\F'(\sqrt[p]{\V_{\F,\emptyset}})/\F')\big),\end{eqnarray}
and 
\begin{eqnarray}\label{equation01} t_\G(\E_\F)=t_\G\big((\E_\F)^\wedge\big)=t_\G\big(\Gal(\F'(\sqrt[p]{\O_\F^\times}/\F')\big).\end{eqnarray}
For all prime ideals $\P_{ij} |\p_i$ of $\O_\F$
we consider the Frobenii $\sigma_{\P_{ij}} $ in  $\Gal(\F'(\sqrt[p]{\V_{\F,\emptyset}})/\F')$. Note that we are no longer in the abelian situation as 
just before Theorem \ref{theo:Gras-Munnier}.
By Theorem \ref{theo:Gras-Munnier} $(i)$, the hypothesis $d(\G_{\F,S})=d(\G_{\F,\emptyset})$ implies 
the Frobenii $\sigma_{\P_{ij}}$ in $\Gal(\F'(\sqrt[p]{\V_{\F,\emptyset}})/\F')$ are without nontrivial relation. As each $\p_i$ splits completely in $\F/\K$, 
we have that 
$\Gal(\F'(\sqrt[p]{\V_{\F,\emptyset}})/\F')$ contains $k$ distinct free $\fq_p[\G]$-modules, one for each $\p_i$. 
The first assertion follows  by (\ref{equation00}). 
  For the second assertion, use the second part of Theorem~\ref{theo:Gras-Munnier} and  (\ref{equation01}). 
  \end{proof}
  
Recall that we denote the abelian group $\prod^d_{i=1}\Z/a_i\Z$ 
by $(a_1,\cdots,a_d)$.
Let  $p=2$ and $\K=\Q(\sqrt{5\cdot 13 \cdot 17 \cdot 29})$ and let $\H=\Q( \sqrt{5},\sqrt{13 },\sqrt{ 17}, \sqrt{ 29})$ be its Hilbert class field.
Here $\Cl_\K=(2,2,2)$, and $\Cl_{\H}=(4,4)$.
Consider the primes $\ell=2311$ and $q=3319$. We easily see $\ell\O_\K ={\mathfrak l}_1{\mathfrak l}_2$ and $q\O_\K = \q_1 \q_2$ and these ideals are all principal.
In the table below we compute the $2$-parts of the ray class groups for  $\K$ and $\H$ of the given conductors.
\begin{table}[h!]
  \begin{center}
    \caption{Ray Class Groups}
    \label{tab:table1}
    \begin{tabular}{c|c|c} 
      \text{Conductor} & $\K$ & $\H$\\
      \hline
         $1$ &  $(2,2,2)$ & $(4,4)$ \\
      ${\mathfrak l}_i, i\in\{1,2\}$ & $(2,2,2)$ & $(4,4)$\\
      $\q_i, i \in \{1,2\}$ & $(2,2,2)$ & $(4,4)$\\
      ${\mathfrak l}_1\q_1$ & $(2,2,2)$ & $(2,2,2,4,4)$\\
      ${\mathfrak l}_1\q_2$ &$(2,2,2,2)$ & $(2,2,2,2,2,4,8)$\\
      ${\mathfrak l}_2\q_1$ & $(2,2,2,2)$& $(2,2,2,2,2,4,8)$\\
      ${\mathfrak l}_2\q_2$ & $(2,2,2)$ & $(2,2,2,4,4)$\\
    \end{tabular}
  \end{center}
\end{table}
The computations were done with MAGMA (see \cite{magma}) and assume the GRH.
\smallskip
Note that in the first three rows, the ray class groups are identical. 
As the principal ideals ${\mathfrak l}_i$ and $\q_i$ split completely in $\H/\K$,  by Lemma \ref{lemma:minkowski-unit-1} one sees that $\O_{\H}^\times \otimes \fq_2$ has a Minkowski unit over $\K$: in other words putting
$\G=\Gal(\H/\K)\simeq (\Z/2\Z)^3$, one has $t_\G(\E_{\H})\geq 1$. 
Note $\H$ is a degree $16$ totally real field so  
$\dim \E_{\H} =16$ and $\E_{\H} \simeq_\G \fq_2[\G] \oplus M$ where $\dim M=8$ so $M$ might be free.
The fourth and seventh rows of the table suggest that $\O_{\H}^\times \otimes \fq_2$ may have two Minkowski unit over $\K$.
The fifth and sixth rows indicate  a relation in the governing extension  between  the Frobenii   of all the primes above ${\mathfrak l}_i$ and $\q_j$ in $\O_{\H}$.

 \smallskip
 
We now show that $M$ is not free.
Set $\K_0=\K$, $\K_1=\Q(\sqrt{5\cdot 17})$, and $\K_2=\Q(\sqrt{13\cdot 29})$. Let $\F$ be the biquadratic field $\K_1 \K_2$.
 Computations show that $\Cl_{\K_1}=(2)$, $\Cl_{\K_2}=(2)$, and $\Cl_\F=(2,4)$.  
Denote by $\varepsilon_i$ the fundamental unit of $\K_i$, and put 
$$e=\# \left(\O_\F^\times/\langle -1, \varepsilon_i, i=0,1,2 \rangle\right).$$
Applying  the Brauer class formula in the biquadratic extension $\F/\Q$, {\it i.e.} $|\Cl_\F|=\frac{1}{4}e |\Cl_{\K}| |\Cl_{\K_1}||\Cl_{\K_2}| $,
to deduce $e= 1$, and then $\O_\F^\times=\langle -1, \varepsilon_i, i=0,1,2 \rangle$.

Let $\sigma$ be a generator of $\G=\G(\F/\K)$.
Observe that:
\begin{enumerate}
 \item[$(i)$] the norm  of $\varepsilon_2$ in $\F/\K$ is $+1$, 
 \item[$(ii)$] the norm of $\varepsilon_1$ in $\F/\K$ is $-1$,
 \item[$(iii)$] $\sigma$ acts trivially on $\varepsilon_0$,
\end{enumerate}
Hence, as we will observe  in Lemma \ref{lemm:minkowski-unit}, one obtains that 
$\E_\F \simeq \fq_2[\G']\oplus \fq_2^2$,
where $\G'=\Gal(\F/\K)$. Finally since   $\E_{\H} \twoheadrightarrow \E_\F$, and  $t_{\G'} (\E_\F) \geq t_\G(\E_{\H})$ we conclude that $ t_\G(\E_{\H})=1$.


\section{Detecting the relations along $\K_\emptyset/\K$}
As mentioned in the remark in the introduction, we can easily find $d$ elements of $\sha^2_\emptyset$ by constructing 
ramified extensions at a low level in the tower $\K_\emptyset/\K$. For $(\G_n)$ a sequence of open normal subgroups of $\G$ with $\displaystyle{ \bigcap^{\infty}_{n=1} \G_n = \{e\}}$, let $\K_n$ be the fixed field of $\G_n$ and set $\H_n=\Gal(\K_n/\K)$.
In this section we show that the presence of torsion elements in the $\fq_p[\H_n]$-module 
$\Gal\left(\K_n\left(\sqrt[p]{ {\mathcal O}^{\times}_{\K_n}}\right)/\K_n\right)$ can give rise to more relations.

\subsection{First observations} \label{section_abelian}
 
 Let $p$ be a prime number and   $\K$ be a number field.
 
 If $\F/\K$ is a Galois extension with Galois group $\G$, the norm  map $N_\G$ sends $\E_\F$ to 
 $\displaystyle \frac{\O_\K^\times}{\O_\K^\times \cap (\O_\F^{\times})^p }\subset \E_\F$; denote 
 by $N'_\G: \E_\F \rightarrow \E_\K$ the  map from $\E_\F$ to $\E_\K$ induced by the norm in $\F/\K$. 
 The commutative diagram:
$$\displaystyle
 \xymatrix{ \E_\F \ar@{->}[r]^-{N_\G} \ar@{->}[rd]_-{N'_\G}& \frac{\O_\K^\times}{\O_\K^\times \cap (\O_\F^{\times})^p}\ \ar@{^{(}->}[r] & \E_\F \\
& \E_\K \ar@{->>}[u]&  }
$$
 implies the following easy lemma:

 \begin{lemm}\label{lemm:comparing_norm_map} One has $N_\G'(\E_\F) \twoheadrightarrow N_\G(\E_\F)$. 
 Moreover,  $\O_\K^\times \cap (\O_\F^{\times})^p =(\O_\K^{\times})^p \implies N_\G'(\E_\F) \simeq N_\G(\E_\F)$.
 \end{lemm}
 
 The study of the norm map $\N_\G$ is "purely algebraic", {}{ i.e. it does not involve number theory}. 
 {} Lemma \ref{lemm_norm_unit} below is proved at the beginning of the proof of  
of  \cite[Lemma 2]{Ozaki}). Since that Lemma is stated differently we include a proof that is essentially from \cite{Ozaki}.
 
 \begin{lemm} \label{lemm_norm_unit}
 Let $\G$ be a finite 
 $p$-group and $M$ an $\fq_p[\G]$-module. 
 Let $\N_\G :M \rightarrow M$ be the norm map.
 Let $m\in M$. Then $\N_\G(m)=0$ if and only if $m$ is a torsion element.
 \end{lemm}
 \begin{proof} Let $0\neq m \in M$.
 Recall  that the annihilator $A_m$ of a nontrivial element $m \in M$ is an ideal of $\fq_p[\G]$.
 
If the annihilator of $m$ is trivial, then its $\fq_p[\G]$-span is isomorphic to  $ \fq_p[\G]$ and so   $\N_\G(\fq_p[\G])=\fq_p$ by (\ref{fpgmodule2}).
 
 Conversely, suppose that $A_m \neq 0$. Then $A_m^\G\neq 0$ since $\G$ is a $p$-group acting on a nontrivial $\fq_p$-vector space. Hence  $A_m^\G \subset (\fq_p[\G])^\G$ which is in turn the one-dimensional vector space $\fq_p \N_\G$. 
 Thus $\N_\G = A_m^\G \subset A_m$ so $\N_\G(m)=0$.
 \end{proof}

  More generally, one has
 
 \begin{theo} \label{theo:abelian_extension} 
 Let $\F/\K$ be a finite $p$-extension with Galois group $\G$ and write  $\N_\G'(\E_\F) \simeq \fq_p^{t}$. 
 Then $t_\G(\E_\F) \leq t \leq t_\G(\E_\F)+d\left( \frac{\O_\K^\times \cap (\O_\F^{\times})^p}{ (\O_\K^{\times})^p} \right)$. 
 In particular if~$\O_\K^\times\cap (\O_\F^{\times})^p= (\O_\K^{\times})^p$, then $t=t_\G(\E_\F)$.
 \end{theo}

 \begin{proof} 
Write 
$\E_\F\simeq \fq_p[\G]^{t_\G(\E_\F)} \oplus \N$, where $\N $ is generated by  torsion elements as an $\fq_p[\G]$-module. 
By  (\ref{fpgmodule2}) and Lemma \ref{lemm_norm_unit}
one has $\N_\G(\E_\F) \simeq  \fq_p^{t_\G(\E_\F)}$. So by Lemma \ref{lemm:comparing_norm_map} we see $ \N_\G'(\E_\F) \simeq  \N_\G(\E_\F)\simeq  \fq_p^{t_\G(\E_\F)} $, proving the result when $\O_\K^\times\cap (\O_\F^{\times})^p= (\O_\K^{\times})^p$.

By noting that the `difference' between $\N_\G(\E_\F)$ and  $\N_\G'(\E_\F)$ is
exactly $\frac{\N_\G(\O_\F^\times)\cap (\O_\F^{\times})^p}{(\O_\K^{\times})^p}$ which has $p$-rank at most  $d\left( \frac{\O_\K^\times\cap (\O_\F^{\times})^p}{(\O_\K^{\times})^p}\right)$, we obtain the general case. 
 \end{proof}
 
 {}


\subsection{Exhibiting relations via the Hochschild-Serre spectral sequence} \label{section:2.2} In this subsection, we flesh out the details of the process described in Remark \ref{Remark:makesha} for explicitly exhibiting $d(\G_\emptyset)$ relations in a minimal presentation of $\G_\emptyset$. It would be helpful to refer to Figure \ref{Fig1}  from the introduction. 
 Put $\K'=\K(\zeta_p)$. Let $S=S_1\cup S_2$ be a set of tame prime ideals of $\O_\K$ such that:
\begin{enumerate}
\item[$-$]  $S_1$ is a minimal set  whose  Frobenii generate the $d(\G_\emptyset)$-dimensional $\fq_p$-vector space $\Gal\left(\K'(\sqrt[p]{\V_\emptyset})/\K'(\sqrt[p]{\O_\K^\times})\right)$,  and 
\item[$-$]   $S_2$ is a minimal set  whose Frobenii generate the $\fq_p$-vector space $\Gal\left(\K'(\sqrt[p]{\O_\K^\times})/\K'\right)$ of dimension $r_1+r_2-1+\delta$.
\end{enumerate}

Recall the Frobenii above are well-defined up to nonzero scalar multiples in the Galois groups, which are vector spaces over $\fq_p$. This ambiguity does not affect their spanning properties.  One has

\begin{lemm} The set $S$ is saturated, in particular $\Sha_S=\{0\}$. Moreover $d(\G_S)=d(\G_\emptyset)$, and $r(\Gg_\emptyset)=d(\X_S)$.
 \end{lemm}

\begin{proof} That $S$ is saturated follows immediately from 
Theorem \ref{theo_saturated}. As there is no dependence relation between  the Frobenii (the set $S$ is minimal),  Theorem \ref{theo:Gras-Munnier} implies $d(\G_S)=d(\G_\emptyset)$. That 
$r(\Gg_\emptyset)=d(\X_S)$ follows from the second part with Lemma  \ref{lemma1}. \end{proof}

\begin{lemm} \label{lemma:deploye} Write  $(a_1,\cdots, a_d)$ for the $p$-part of  $\RCG{\K}{\emptyset}$ and let $S_1=\{\p_1,\cdots, \p_d\}$ as above. Then
$\RCG{\K}{\p_1,\cdots, \p_d} \twoheadrightarrow (pa_1,\cdots, pa_d)$.
\end{lemm}

\begin{proof}
This is a consequence of Theorem \ref{theo:Gras-Munnier}. As the primes of $S_1$ split completely in the governing extension $\Gal(\K'(\sqrt[p]{\O_\K^\times})/\K')$,  for each prime ideal $\p \in S_1$ we have 
$\#\G_{\{\p\}}^{ab} \neq \# \G_\emptyset^{ab}$. We conclude by noting that $d(\G_{S_1})=d(\G_\emptyset)$.
\end{proof}

Lemma \ref{lemma:deploye} implies the existence of  $d$ independent degree-$p$ cyclic extensions $\F_i$ of $\K_\emptyset$, each totally ramified at $\p_i$, $i=1,\cdots, d$, and on which $\G_\emptyset$ acts trivially, implying that $d( \X_S) \geq d$.
{\it These additional relations  are accessible via the set  $S_2$}.

  \subsection{Proof of Theorem A} \label{section:alongthetower}
  
 Let $(\G_n)$ be a  sequence of open normal subgroups 
 of $\G_\emptyset$ such that $\G_n \subset \G_{n+1}$ and  $\bigcap_n  \G_n=\{e\}$.
 Put $\H_n:= \G_\emptyset/\G_n$,  $\K_n:=\K_\emptyset^{\G_n}$, and write
 $\E_{\K_n}:=\fq_p[\H_n]^{t_n}\oplus \N_n$ where $\N_n$ is torsion as an $\fq_p[\H_n]$-module.
  
 \begin{lemm}\label{lemm:t_n} The sequence $(t_n)$ is nonincreasing.
 \end{lemm}
 
\begin{proof}
Recall from (\ref{fpgmodule1}) 
that the norm  map from ${\H_{n+1,n}}:=\Gal(\K_{n+1}/\K_n)$  on $\fq_p[\H_{n+1}]$ induces the following  $\fq_p[\H_n]$-isomorphisms: 
$$\fq_p[\H_n]\simeq \fq_p[\H_{n+1}]_{\H_{n+1,n}} \simeq \fq_p[\H_{n+1}]^{\H_{n+1,n}}.$$ 
The norm map $N_{{\H_{n+1,n}}}$ of $\K_{n+1}/\K_n$ induces a morphism from $\E_{\K_{n+1}}$ to $\E_{\K_{n+1}}$ which allows us to obtain $$\fq_p[\H_{n}]^{t_{n+1}} \hookrightarrow \frac{\O_{\K_n}^\times}{\O_{\K_n}^\times \cap (\O_{\K_{n+1}}^{\times})^p} \twoheadleftarrow \E_{\K_n},$$ which implies $t_n \geq t_{n+1}$.
\end{proof}

\begin{defi} Set $\lambda:= \lambda_{\K_\emptyset/\K} =\lim_n t_n$. We call this the  {\it Minkowski-rank} of the units along $\K_\emptyset/\K$. 
\end{defi}

One easily sees that $\lambda$ does not depend on the sequence $(\G_n)$.

\medskip

Let us write $p^\beta:=[\K'(\sqrt[p]{\O_\K^\times})\cap \K'\K_\emptyset:\K']=\big[\O_\K^\times \cap (\O_{\K_\emptyset}^{\times})^p : (\O_\K^{\times})^p\big]$.
 Obviously, $\beta \leq \min(d(\O_\K^\times), d(\G_\emptyset))$.
 
 \begin{prop} \label{prop:delta_beta}
One has:    $\delta=0 \Longrightarrow \beta=0$.
 \end{prop}

 \begin{proof}
  Let $\Delta=\Gal(\K'/\K)$ be the Galois group of $\K'/\K$; by hypothesis $\Delta$ is  of order coprime to $p$.
  As $\Delta$ acts trivially on $\O_\K^\times$, by Kummer duality the action of $\Delta$ over $\Gal(\K'(\sqrt[p]{\O_\K^\times})/\K')$ is given by the cyclotomic character; in particular, there is no nontrivial subspace of $\Gal(\K'(\sqrt[p]{\O_\K^\times})/\K')$ on which  $\Delta$ acts  trivially.
  But $\Delta$ acts trivially on $\Gal(\K'\K_\emptyset /\K')$; the result holds. 
 \end{proof}

\medskip

\begin{theo} \label{theo:chi} We have the estimates:
 $$ d(\O_\K^\times)-\lambda - \beta \leq \df (\G_\emptyset)  \leq d(\O_\K^\times) - \lambda.$$
In particular, 
\begin{itemize}
\item if $\O_\K^\times \cap (\O_{\K_\emptyset}^{\times})^p=(\O_\K^{\times})^p$ or if $\delta=0$, then 
$\df(\G_\emptyset)= d(\O_\K^\times)- \lambda$.
\item if $\lambda = d(\O_\K^\times)$ then $\df(\G_\emptyset)=0$.
\end{itemize}
\end{theo}

\medskip

\begin{proof} 
We keep the notations of the beginning of the section.

{} 
We give two proofs for the lower bound. The first one is `algebraic' while the second is number-theoretic and is more `explicit'  in how  we determine the existence of the relations.

\smallskip
We first establish  the upper bound.
{} 
Denote by $N_{\H_n}$ the norm map for the extension $\K_n/\K$. Observe that by Corollary \ref{prop:minoration_relations}, $\df(\G_\emptyset)= d(\E_\K)-d(N_{\H_n}'(\E_{\K_n}))$ for $n\gg 0$. Take   $n$ sufficiently large  such that $t_n=\lambda$.
One has $d(N_{\H_n}(\E_{\K_n}))\geq \lambda$ (see Theorem \ref{theo:abelian_extension}), implying that   $d(N_{\H_n}'(\E_{\K_n})) \geq \lambda$. Hence  one gets:   $$\df(\G_\emptyset) \leq d(\O_\K^\times) - \lambda.$$

{} Below are the two proofs of the lower bound.

$\bullet$ First proof: 

Observe that $\beta=  d\left(\frac{\O_\K^\times \cap (\O_{\K_n}^{\times})^p  }{ ( \O_\K^{\times})^p }\right)$ since $n\gg 0$.
By Theorem \ref{theo:abelian_extension} one  also has  $d(N_{\H_n}'(\E_{\K_n}))\leq \lambda +\beta$, and then by Corollary \ref{prop:r-d_wingberg} we  get $$\df(\G_\emptyset) \geq d(\O_\K^\times) - \lambda -\beta.$$

\smallskip

$\bullet$ Second proof:

{} Here we show $d(\O_\K^\times)  - \lambda -\beta \leq \df(\G_\emptyset)$ using saturated sets and the Hochschild-Serre exact sequence. {}

{}

First assume that $\zeta_p \in \K$ {\it i.e.} $\delta=1$.  Choose $n\gg 0$, and  
write  $\E_{\K_n}=\fq_p[\H_n]^\lambda\oplus \N_n$,  where $\N_n$ is a $\H_n$-torsion $\fq_p[\H_n]$-module. 

\begin{figure}[h]
{
$$\xymatrix{ 
 {\rm R}_i \K_\emptyset & &  &\\
& \K_\emptyset \ar@{-}[ul]& & \\
&&\L_5 :=\K_n\left(\sqrt[p]{\O_{\K_n}^\times}\right)  & \\
 {\rm R}_i \ar@{-}[uuu]& & \L_4:= \K_n\left(\sqrt[p]{\O_\K^\times}\right)  \ar@{-}[u] & \\
& \K_n \ar@{-}[uuu] \ar@{-}[ur]. \ar@{-}[ul]&   &  \L_3:=\K\left(\sqrt[p]{\V_\emptyset}\right)  \ar@{-}[ld]\\
& & \L_2:=\K\left(\sqrt[p]{\O_\K^\times}\right)   \ar@{-}[uu] &\\
& \L_1:=\K_n\cap \K\left(\sqrt[p]{\O_\K^\times}\right) \ar@{-}[uu] \ar@{-}[ur]& &\\
& \K \ar@{-}[u] & & 
}$$
}
\caption{}\label{Fig2}
\end{figure}

Put $\E_\K':=\frac{\O_\K^{\times}}{\O_\K^\times \cap (\O_{\K_n}^{\times})^p }\hookrightarrow \E_{\K_n}$.

Hence $[\L_4:\K_n]=\#\E_\K'$.  
Observe that $\Gal(\L_4/\K_n) \simeq \fq_p^t$, where $t =d(\O_\K^\times)- \beta$.

We will find  a set of primes of $S=\{ \p_1,\cdots \p_{t-\lambda}\}$ in $\K$
such that:
\begin{itemize}
\item $\p_i$ splits completely in $\K_n/\K$,
\item Their Frobenii span a $(t-\lambda)$-dimensional space in $\Gal(\L_2/\L_1) \simeq \Gal(\L_4/\K_n)$,
\item For each $i$, let $\b_{ij}$ be the primes above $\p_i$ in $\K_n$. There is a dependence relation on the Frobenii of the $\b_{ij}$ in $\Gal(\L_5/\K_n)$. By Gras-Munnier (Theorem \ref{theo:Gras-Munnier}) this implies the existence of a $\Z/p$-extension 
${\rm R}_i/\K_n$ ramified only at (these primes above) $\p_i$.
Let ${\tilde{\rm R}}_i$ be the Galois closure over $\K$  of ${\rm R}_i$. As the $p$-group $\Gal(\K_n/\K)$ must act on 
the $\fq_p$-vector space $\Gal( {\tilde{\rm R}}_i/\K_n)$ with a fixed point, by iteration we may assume ${\rm R}_i/\K$ is Galois.
 The $\Z/p$-extension ${\rm R}_i\K_\emptyset$ is ramified only at $\p_i$ and gives an element of $\sha^2_\emptyset$.
We have produced  $t-\lambda$ elements of $\sha^2$ {\em in addition} to the $d$ elements of $\sha^2_\emptyset$ we get by choosing primes $\{\q_1,\cdots,\q_d\}$ of $\K$ whose Frobenii form a basis of $\Gal(\L_3/\L_2)$.
\end{itemize}
This gives the lower bound. We now construct $S$.

 As $\L_2/\K$ is abelian ($\zeta_p\in \K$), $\H_n = \Gal(\K_n/\K)$ acts trivially on $\Gal(\L_2/\L_1)$ (and thus on
 $\Gal(\L_4/\K_n)$ as well).
 
After taking the Kummer dual of $\E_{\K_n}$, one obtains $\Gal(\L_5/\K_n)\simeq \fq_p[\G]^\lambda \oplus \M$,  where $\M_n=\N_n^\vee$ is a $\H_n$-torsion $\fq_p[\H_n]$-module. The natural surjection
$\pi:\Gal(\L_5/\K_n) \twoheadrightarrow \Gal(\L_4/\K_n)$ induces, upon
taking  $\H_n$-coinvariants, the map $$\Gal(\L_5/\K_n)_{\H_n}\simeq \fq_p^\lambda\oplus (\M_n)_{\H_n} \stackrel{\pi}{\twoheadrightarrow} \Gal(\L_4/\K_n) \simeq \Gal(\L_2/\L_1) \simeq \fq_p^t.$$ Thus $$d(\pi((\M_n)_{\H_n})) \geq t-\lambda=d(\O_\K^\times)- \beta - \lambda.$$

Take (at least) $t-\lambda$ elements $x_i$ in $\Gal(\L_5/\K_n)$, such that their image under the projection $\pi$ forms a basis of $\pi((\M_n)_{\H_n}) \subset
\Gal(\L_4/\K_n) \simeq \Gal(\L_2/\L_1) \simeq \fq_p^t.$ We choose $\p_i$ to split completely in $\K_n/\K$ and have Frobenius $\pi(x_i) \in \Gal(\L_2/\L_1)$, so clearly $\p_i$ satisfies the first two points above. We have chosen $\p_i$ so that the primes above it in $\K_n$ have Frobenii generating a $\fq_p[\H_n]$-torsion module in $\Gal(\L_5/\K_n)$. This settles the third point and the case $\delta=1$.

\smallskip

Suppose now that $\delta=0$. 
Replace every field ${\rm E}$ above by ${\rm E}':={\rm E}(\zeta_p)$.
 The key fact is this: by Proposition \ref{prop:delta_beta}, one has $d(\Gal(\L'_2/\L'_1))=d(\O_\K^\times)$ so 
$\L'_2 \cap \K_n = \K$.
The rest of the proof is word for word the
same from this point on.

\smallskip
The last result follows since $\df(\G_\emptyset)\geq 0$. 
\end{proof}

\begin{rema}
Observe that
\begin{enumerate}
\item[$(i)$] the inequality $\df (\G_\emptyset)\leq  d(\O_\K^\times)  - \lambda$ comes from universal norms of units and is Wingberg's result (Theorem \ref{theo:wingberg});
\item[$(ii)$]  the group $\G_\emptyset$ has at least $\lambda$ {\it fewer} relations than the maximal possible number, $\dim \V_\emptyset/ (\K^{\times})^p$.
\end{enumerate}
\end{rema}

\begin{coro}
Suppose $\K_\emptyset/\K$ is finite. Then $\lambda < d(\O_\K^\times) - d^2/4+d$.
 \end{coro}

  \begin{proof}
  By the Theorem of Golod-Shafarevich one has $\df(\G_\emptyset)> d^2/4-d$; then  apply Theorem \ref{theo:chi}.
  \end{proof}

  \subsection{Remarks}
  
  \subsubsection{When $\G_\emptyset$ is abelian} \label{section:abelian}
  
  \medskip
  
  $\bullet$ Consider first the case where $\G_\emptyset$ is cyclic. Clearly $d(\G)=r(\G)=1$ so $\df(\G_\emptyset)=0$.
  By Theorem \ref{theo:abelian_extension}, we get 
  $$ \lambda= t_{\G_\emptyset}(\E_{\K_\emptyset}) \geq d(\O_\K^\times)-\beta \geq d(\O_\K^\times)-1,$$
  due to the fact that $\beta \leq 1$.
  In particular, this situation forces $\K_\emptyset$ to have a Minkowski unit provided $\K$ is neither $\Q$ nor imaginary quadratic.
We can recover this fact by using the well-known following result: as $\G_\emptyset$ is cyclic, every element of $\O_\K^\times$ is the norm of an element of $\O_{\K_\emptyset}^\times$.
 Note this last argument applies in the quadratic imaginary case as well. 
  
As an example, take the imaginary quadratic number field $\K=\Q(\sqrt{-q\cdot \ell})$, with $-q \equiv \ell \equiv1 ({\rm mod} \ 4)$.  Here, $p=2$, $\G_\emptyset$ is cyclic, and $\O_\K^\times\cap \O_{\K_\emptyset}^{\times 2}= \O_\K^{\times 2}$.
  We find $\lambda=1$, and finally that $\E_{\K_\emptyset}\simeq \fq_2[\G_\emptyset]$.

 Observe that if $\G_\emptyset \simeq \Z/2\Z$, then the fundamental unit of the biquadratic extension $\K(\sqrt{\ell})$ is exactly the fundamental unit of the quadratic field $\Q(\sqrt{\ell})$ and then is of norm~$-1$: hence, as we will see later, in this case we find again $\E_{\K_\emptyset} \simeq \fq_2[\G_\emptyset]$.
 
 \medskip

  $\bullet$ Take $p=2$, and $\K$ such that $\G_\emptyset \simeq (\Z/2\Z)^2$. 
Here $d(\G)=2$ and $r(\G)=3$ so $\df(\G_\emptyset) =1$, implying the existence of an extra relation.
  By Theorem \ref{theo:abelian_extension}, we get
  $$\lambda= t_{\G_\emptyset}(\E_{\K_\emptyset}) \geq  d(\O_\K^\times)-2,$$
  due to the fact that $\beta\leq 2$.
  
  Let us  be more precise: Kisilevsky in  \cite{Kisilevsky} showed that  if $\G_\emptyset \simeq (\Z/2\Z)^2$, then for every quadratic subextension $\F_i/\K$ in $\K_\emptyset/\K$, one has $\big(\O_\K^\times: N_{\F_i/\K}\O_{\F_i}^\times\big)=2$. We prove

  \begin{prop}
   Let $\K/\Q$ be a quadratic extension such that $\G_\emptyset \simeq (\Z/2\Z)^2$. Then the extra relation is detected at one of the three quadratic subextensions $\F/\K$ in $\K_\emptyset/\K$. 
  \end{prop}

  \begin{proof}
  Suppose first that $\K/\Q$ is an imaginary quadratic extension. By Kisilevsky's result  $-1$ is not a norm of any unit in each of  the three subextensions $\F_i/\K$ of $\K_\emptyset/\K$. Let us choose $\F:=\F_i$ such that $\F \neq \K(\sqrt{-1})$; put $\G=\Gal(\F/\K)$.  By using $N_\G(\O_\F^\times) \subset \{\pm 1\}$, it is then easy to see that, modulo $\O_\F^{\times 2}$, $-1$ is not a norm of any unit in $\F/\K$ implying that the norm map $N_{\G} :\E_\F \to \frac{\O_\K^\times}{\O_\K^\times \cap \O_\F^{\times 2}}$ is not onto. 
  
Recall that $\frac{\O_\K^\times}{\O_\K^\times \cap\O_\F^{\times 2}} \hookrightarrow \E_\F$. 
 As $\dim_{\fq_2} \E_\F=2$,  the only possibilities for the $\fq_2[\G]$-module $\E_\F$ are $\fq_2^2$ and $\fq_2[\G]$. As the norm map is onto in the latter case we see $\E_\F \simeq \fq_2^2$ and
 then $t_\G(\E_\F)=0$ so $\lambda=0$:
  the extra relation is detected by the quadratic extension $\F/\K$.
  
  \smallskip
  
We now settle the case where $\K/\Q$ is real quadratic extension. Denote by $\varepsilon$ the positive fundamental unit of~$\K$. By Kisilevsky's result, one knows that  for every quadratic subextension $\F_i/\K$,  $-1$  or $\varepsilon$  is not a norm of any unit in $\F_i/\K$.
  Take one such quadratic extension $\F/\K$, and put $\G=\Gal(\F/\K)$. 
  
  Suppose that $-1$ is not a norm of from $\F$ to $\K$ of any unit  but    $-1 \in N_\G(\O_\F^\times) \O_\F^{\times 2}$. First, $N_\G(\O_\F^\times) \subset \{1, \pm \varepsilon\}$ modulo squares.
The equations $-1=  z^2$ and $-1=\varepsilon z^2$ have no solutions with $z\in \O_\F^\times$ for sign reasons. Hence the only possible solution is that  $-1=-\varepsilon z^2$, and then, necessarily $\F=\K(\sqrt{\varepsilon})$.
  
  Suppose now that  $\varepsilon$  is not a norm of any unit in $\F/\K$. 
  As before, if we test the condition $\varepsilon \in N_\G(\O_\F^\times) \O_\F^{\times 2}$, we see the equations $\varepsilon =-z^2$, and $\varepsilon=-\varepsilon^a z^2$ have no solution for sign reasons. 
  Suppose that $\varepsilon=\varepsilon^a z^2$ for some odd integer~$a$ 
  with $\varepsilon^a \in N_\G(\O_\F^\times)$. As $N_\G(\varepsilon)=\varepsilon^2$, it is easy to see this
   implies  $\varepsilon \in N_\G(\O_\F^\times)$, which  contradicting our assumption. Thus  $a$ is even. 
   and we conclude that $\varepsilon \in \O_\F^{\times 2}$, {\it i.e.} $\F=\K(\sqrt{\varepsilon})$.
  
  Hence, in any quadratic subextension $\F/\K$ in $\K_\emptyset/\K$ such that $\F\neq \K(\sqrt{\varepsilon})$, one has 
  that the map
  $N_{\G} :\E_\F \to \frac{\O_\K^\times}{\O_\K^\times \cap \O_\F^{\times 2}}$ is not onto, and the result follows as in the imaginary case.
  \end{proof}

\subsubsection{Remarks on $\beta$} 
Let $d\in \Z_{>0}$ square free, and let $\K=\Q(\sqrt{d})$ be a real quadratic  field with fundamental unit $\varepsilon$. One wants to test if $\varepsilon \in \O_{\K_\emptyset}^{\times 2}$. 
 This will imply $\beta >0$. 
Obviously, we assume  $\N_{\K/\Q} \varepsilon=1$. 
Let us write $\delta(\varepsilon)=2+\Tr_{\K/\Q}( \varepsilon) \in \Z$ where $\Tr_{\K/\Q}$ is the trace map; this quantity has been introduced by Kubota in \cite{Kubota}.

Recall that, up to square of $\Q$, the quantity $\delta(\varepsilon) $ divides  $\disc(\K)$ (see \cite[Hilfssatz 8]{Kubota}).

Moreover  by \cite[Hilfssatz 11]{Kubota},  given a real biquadratic extension $\F/\Q$ containing $\K$, then $\sqrt{\varepsilon}\in \F$ if and only if, $\sqrt{\delta(\varepsilon)}\in \F$.  In particular as easy consequence we get:

\begin{lemm} \label{coro_criteria_squareunit_ramified}
Let $\K=\Q(\sqrt{d})$ be a real quadratic field. Let $\varepsilon $ be the fundamental unit of $\K$ that we suppose of norm $+1$. Then $\K(\sqrt{\varepsilon})/\K$ is unramified if and only if, $\K(\sqrt{\delta(\varepsilon)})/\K$ is unramified.

In particular, when $d\equiv 3 \ ({\rm mod} \ 4 )$,
then  $\O_\K^\times \cap \O_{\K_\emptyset}^2=\O_\K^{\times 2}$ if and only if the $2$-valuation $v_2(\delta(\varepsilon))$ of $\delta(\varepsilon)$ is odd.
\end{lemm}

Take $\K=\Q(\sqrt{d})$  as in the following table; let us write $d=d_0 \cdot q$, where $d_0$ is fixed and where $q$ is some varying prime number.
For $X\geq 1$, put $A(X)=\#\{q \equiv 3 ({\rm mod} \ 4), q \leq X \}$, and $B(X)=\{ q \in A(X), v_2(\delta(\varepsilon)) \equiv 1 ({\rm mod} \  2)\}$, where $\varepsilon$ is the fundamental unit of the quadratic field $\K$. We compute the proportion $\#B(X)/\#A(X)$: that is the proportion of quadratic real fields $\K=\Q(\sqrt{d})$ for which $\O_\K^\times \cap \O_{\K_\emptyset}^2=\O_\K^{\times 2}$.
$$\begin{array}{c|c|c|c}
d &q\leq 10^5&  10^6 & 10^7 \\
\hline 
3 \cdot 5 \cdot 7 \cdot q  &2604/4806 \approx .5418 &.5428&.5416 \\
3 \cdot 5 \cdot 11 \cdot q & .5429 & .5401 & .5421\\
5\cdot 7 \cdot 11 \cdot q& .6271 &.6252 &.6251 \\
3 \cdot 7 \cdot 13 \cdot q & .6317&.6267 & .6249\\
3 \cdot 17 \cdot 7 \cdot q & .4621& .4593&.4784 \\
3 \cdot 5 \cdot 7 \cdot 11 \cdot 19 \cdot q  & .5980 &.5945 & .5936 \\
\end{array}
$$

We can do  more computations: the next table  indicates the proportion of quadratic real  fields $\K=\Q(\sqrt{d})$ for wich $\O_\K^\times \cap \O_{\K_\emptyset}^2=\O_\K^{\times 2}$:

$$\begin{array}{c|c|c|c|c}
d &\leq 10^5&  10^6 & 10^7& 10^8 \\
\hline 
d\equiv 3 ({\rm mod}  4) &.7170 & .6948&.6885 &.6659\\
\end{array}
$$

 


\medskip

\section{Consequences}

Throughout this section, we explore consequences of the previous results, including:
\begin{itemize}
\item How $\lambda$ and deficiency change as we move up the tower $\K_\emptyset/\K$;
\item That $\df(\G_\emptyset)=0$ implies the same for open subgroups of $\G_\emptyset$ when $\delta =1$;
\item The rapid growth of $\lambda$ as we move up a $p$-adic analytic quotient tower of $\G_\emptyset$. The Tame Fontaine-Mazur conjectrure predicts that infinite $p$-adic analytic quotients of $G_\emptyset$ do not exist; thus, proving $\lambda$ cannot grow rapidly would lend support to the Fontaine-Mazur conjecture;
\item Some results in the direction of better understanding the cohomological dimension of $\G_\emptyset$;

\item A computable test for maximality of $\df(\G_\emptyset)$. 
\end{itemize}

  \subsection{Conserving the  deficiency along the tower}
  
 Let $\F$ be a number field in the tower $\K_\emptyset /  \K$ and recall  that $\F_\emptyset=\K_\emptyset$.
We denote by $\lambda_{\F_\emptyset/\F}$ the asymptotic Minkowski rank in $\F_\emptyset/\F$.

\begin{prop} \label{prop:borneinflambda}
One has $\lambda_{\F_\emptyset/\F} \geq [\F:\K] \lambda_{\K_\emptyset/\K}$.
\end{prop}
  
  \begin{proof}
  Let $\L\supset \F \supset \K$  in $\K_\emptyset/\K$ be a large enough number field so that
 $\lambda_{\L/\K} = \lambda_{ \K_\emptyset/\K}$ and $\lambda_{\L/\F} = \lambda_{ \F_\emptyset/\F}$. 
Set  $\G=\Gal(\L/\K)$ and   $\H=\Gal(\L/\F)$. Then $\E_\L=\fq_p[\G]^{\lambda_{\K_\emptyset/\K}} \oplus \N$, where $\N$ is $\G$-torsion. The result follows by noting that $\fq_p[\G] \simeq_\H \fq_p[\H]^{[\F:\K]}$ (see \S \ref{subsection-algebraic_tools}).
  \end{proof}

  \begin{coro}For every number field $\F$ in $\K_\emptyset/\K$, we have $$\df(\G_{\F,\emptyset})\leq  d(\O_\F^\times) -[\F:\K] \lambda_{\K_\emptyset/\K}.$$
  \end{coro}

  \begin{proof}
  This follows immediately from Theorem \ref{theo:chi} and Proposition \ref{prop:borneinflambda}.
  \end{proof}

  \begin{rema} The above Corolllary is very close to what one could expect from strictly group-theoretic considerations. Namely, from equations $(5.2)$ and $(5.4)$ of \cite{Koch} one deduces 
that for an open subgroup $\H$ of a pro-$p$ group $\G$, one has $$\df(\H)+1 \leq (\G:\H)\big(\df(\G)+1\big).$$
  \end{rema}

  \medskip
  
  \subsection{When $\df(\G_\emptyset)=0$}

\begin{coro} \label{theo:r-d} Let $\K$ be a number field containing $\zeta_p$.
Suppose that  $\O_\K^\times \cap (\O_{\K_\emptyset}^{\times})^p=(\O_\K^{\times})^p$, and that  $\df(\G_\emptyset)=0$.
  Then, for every finite extension $\F/\K$ in $\K_\emptyset/\K$, one has  $\df(\G_{\F,\emptyset})=0$.
\end{coro}

\begin{proof} 
Applying Theorem \ref{theo:chi}, we see 
 $\lambda = d(\O_\K^\times)$ and is maximal and hence constant in the tower $\K_\emptyset/\K$, relative to the base field $\K$. 
 By Proposition ~\ref{prop:borneinflambda} we see 
 $$\lambda_{\F_\emptyset/\F} \geq [\F:\K]\lambda = [\F:\K] d(\O_\K^{\times }) =d(\O_\F^\times).$$ The result follows by Theorem \ref{theo:chi}.
\end{proof}

\begin{coro} \label{coro:r=d_quadratic} Let $\G$ be a pro-$2$ group such that:
\begin{enumerate}
\item[$(i)$] $\df(\G)=0$,
\item[$(ii)$] there exists a normal open  subgroup
$\H$ of $\G$ such that $r(\H)\neq d(\H)$.
\end{enumerate}
Then $\G$ cannot be realized as the $2$-tower of an imaginary quadratic field $\K$ of discriminant $\disc_\K \equiv 1 ({\rm mod} \ 4)$ nor $\disc_\K \equiv 0 ({\rm mod} \ 8)$.
\end{coro}

\begin{proof} The discriminant hypotheses imply
 $-1 \notin \O_{\K_\emptyset}^{\times 2}$. The result follows from  Corollary \ref{theo:r-d}.
\end{proof}

The condition that the number of Minkowski units is maximal is  very strong \footnote{We thank Ozaki for bringing this result to our attention.}:

\begin{prop} Let  $\G$ be a finite $p$-group  such that $\df(\H)=0$ for every subgroup $\H$ of~$\G$. Then $\G$ is cyclic, or the generalized quaternion group $Q_{2^n}=\langle x,y \  | \ x^{2^{n-1}}=1, x^{2^{n-2}}=y^2, yxy^{-1}=x^{-1}\rangle$, 
$n\geq 3$.
\end{prop}

\begin{proof} 
In this case every abelian subgroup $\H$ of $\G$ is of deficiency $0$, forcing  $\H$  to be cyclic. Then, $\G$ is cyclic or the generalized quaternion group $Q_{2^n}$ of order $2^n$ (see for example \cite[Theorem 9.7.3]{Scott}).
For the converse,   obviously cyclic groups $\G$ satisfy  $\df(\H)=0$ for every subgroup $\H$ of $\G$. 
Concerning $Q_{2^n}$, recall that its subgroups  are cyclic or isomorphic to $Q_{2^{n-1}}$, and that  the Schur multiplier of the  generalized quaternion groups $Q_{2^k}$ are trivial (or in other words that $\df(Q_{2^k})=0$).
\end{proof}

\begin{rema} Take $p=2$, and
let $\K$ be an imaginary quadratic field.
Recall that $\df(\G_\emptyset) \in \{ 0,1\}$.
We suspect that when $\G_\emptyset$ is infinite then $\df(\G_\emptyset)$ is maximal.
In fact, if it is not  the case, then by Theorem \ref{theo:r-d}, we get that $r(\H)=d(\H)$ for every open distinguished subgroup $\H$ of $\G_\emptyset$.
\end{rema}

\begin{rema} Observe that
 Poincar\'e pro-$p$ groups of dimension $3$ satisfy condition of Corollary \ref{theo:r-d}, see for example \cite[Chapter III, \S 7]{NSW}.
\end{rema}

\smallskip
We close this subsection with an explicit, albeit contrived, example with $p=2$.
\begin{Example} Let $\K=\mathbb Q (\sqrt{ -3\cdot 5\cdot 53})$. An easy MAGMA computation gives that the class group of $\K$ is $(2,2)$ and its $2$-Hilbert Class Field tower has order $8$. Some straightforward computations show this group has at least three cyclic subgroups of order $4$, hence it is the quaternion group of order $8$.
Here $\O_\K^\times =\{1,-1\}$, and 
as the discriminant of $\K$ is prime to $4$,
$i^2=-1 \notin \O_{\K_\emptyset}^{\times 2}$ so 
$\O_\K^\times \cap \O_{\K_\emptyset}^{\times 2} =\{1\}$ 
and
$\beta =0$. Then Theorem \ref{theo:chi} gives $\df(\G_\emptyset) =1-\lambda$. But it is well-known the quaternion group has deficiency $0$ so $\lambda=1$. There is a Minkowski unit in this (short) tower.
\end{Example}

\subsection{In the context of the Fontaine-Mazur conjecture}

The conjecture of Fontaine-Mazur \cite[Conjecture 5a]{FM} asserts that every analytic quotient of $\G_\emptyset$ must be finite. By class field theory, one knows that every infinite analytic quotient of $\G_\emptyset$ must be of analytic dimension at least $3$ (see \cite[Proposition 2.12]{Maire-MMJ}).

One knows that $\G_\emptyset$ is not $p$-analytic when the $p$-rank $d(\Cl_\K)$ of the class group $\Cl_\K$ of $\K$ is large  compared  to $[\K:\Q]$. See
A.3.11 of \cite{Lazard}. Alternatively, this is (literally!) an exercise on page 78 of \cite{Serre}.

\medskip

Suppose $\G:=\G_\emptyset$ is  infinite and analytic. One knows that every infinite analytic pro-$p$ group contains an open {\it uniform} subgroup. To simplify, assume $\G$ is uniform.  Denote by $(\G_n)$ the $p$-central descending series of   $\G$ (it is also the Frattini series), and let $\K_n=\K_\emptyset^{\G_n}$. Put $\H_n=\Gal(\K_n/\K)$; recall that $\# \H_n=p^{dn}$, where $d=d(\G_\emptyset)$ is also the dimension of $\G_\emptyset$ as analytic group. For $n\geq 1$, denote by $\lambda_n$ the Minkowski-rank of the units along $\K_\emptyset/\K_n$.

 \smallskip
  The hypothesis of Corollary \ref{coro:growth} below is, assuming the Fontaine-Mazur Conjecture, never satisfied. We include the Corollary to indicate a possible strategy to prove $\G_\emptyset$ is not analytic, namely show the number of Minkowski units does {\em not} grow so rapidly in the tower.

 \begin{coro} \label{coro:growth} Let $\G_\emptyset$ be pro-$p$ analytic of dimension $d$. Then for $m$ large,
 $$ (r_1+r_2)[\K_m:\K]-1+\delta-\frac{d(d-1)}{2} \leq \lambda_m \leq  (r_1+r_2)[\K_m:\K]-1+\delta-\frac{d(d-3)}{2}.$$
 \end{coro}

 \begin{proof}
Theorem \ref{theo:chi} here, Theorem $4.35$ of \cite{Dixon}, and  the assumption that $\G_\emptyset$ is uniform imply  $\df(\G_m)$ 
is constant and equal to $\displaystyle{\df(\G_\emptyset)=\frac{d(d-3)}{2}}$. As remarked in the introduction,
$\beta_m \leq d(\G_m)=d(\G)=d$. We immediately see $\lambda_m \sim (r_1+r_2)[\K_m:\K]$, proving the main terms in the estimates.
\smallskip

We now prove the more refined estimates.
Let us choose $n\gg m \gg 0$ such that:
$$\E_{\K_n}=\fq_p[\H_{n,m}]^{\lambda_m}\oplus \N_{n,m}= \fq_p[\H_n]^\lambda\oplus \N_n$$
where $\lambda=\lambda_{\K_\emptyset/\K}$, $\lambda_m=\lambda_{\K_\emptyset/\K_m}$,  $\H_{n,m}=\Gal(\K_n/\K_m)$,  
and  $\N_{n,m}$ and $\N_n$  are torsion modules over  $\fq_p[\H_{n,m}]$ and $\fq_p[\H_{n}]$ respectively.

Then by Proposition \ref{prop:borneinflambda}, we see  $\lambda_m = [\K_m:\K] \lambda+\lambda_{m,n}^{new}$;
the quantity  $\lambda_{m,n}^{new}$ corresponds to the $\fq_p[\H_{n,m}]$-free part in $\N_n$.
Hence, by Theorem \ref{theo:chi}, one has
$$\df(\G_m)\geq d(\O_{\K_m}^\times) -\lambda_m-\beta_m \geq
d(\O_{\K_m}^\times) - [\K_m:\K] \lambda - \lambda_{m,n}^{new}-d(\G_m).$$
After noting that $d(\O_{\K_m}^\times)=(r_1+r_2)[\K_m:\K]-1+\delta$, we get
\begin{eqnarray} \label{identite0}
\df(\G_m)\geq (r_1+r_2-\lambda)[\K_m:\K] -1 + \delta -d -\lambda_{m,n}^{new}
\end{eqnarray}
But
$\df(\G_m)=\displaystyle{\df(\G_\emptyset)=\frac{d(d-3)}{2}}$. 
Hence  (\ref{identite0}) becomes
\begin{eqnarray}\label{identite} \lambda_{m,n}^{new}\geq (r_1+r_2-\lambda)[\K_m:\K]-1+\delta  -\frac{d(d-1)}{2}.
\end{eqnarray}  
and 
\begin{eqnarray}\label{identite1}  \lambda_m \geq (r_1+r_2)[\K_m:\K]-1+\delta-\frac{d(d-1)}{2},
\end{eqnarray}
proving the first inequality.
The upper bound follows as $\df(\G_m) \leq d(\O_{\K_m}^\times) -\lambda_m$ so
\begin{eqnarray}\label{identite2}
\lambda_m \leq (r_1+r_2)[\K_m:\K]-1+\delta - \frac{d(d-3)}{2}.
\end{eqnarray}
 \end{proof}

\medskip

\subsection{On the cohomological dimension of $\G_\emptyset$}

Since the works of Labute \cite{Labute}, Labute-Min\'a\v{c} \cite{Labute-Minac} and Schmidt \cite{Schmidt}, etc. one knows that  in certain cases the groups  $\G_S$,  for $S$ tame, \
are of cohomological dimension $2$.  In all the examples of these papers  $S \neq \emptyset$. 
The question of the computation of cohomological dimension of $\G_\emptyset$ is still an open problem (one can find partial negative answers in \cite{Maire-cupproducts}). 
To prove Theorem \ref{theo:coho_dimension}, we need the following  lemma due to  Schmidt~\cite[Proposition 1]{Schmidt-BLMS}.
 
 \begin{lemm} \label{prop:dim_coho_schmidt} (Schmidt) Let $\G$ be an infinite pro-$p$ group  such that for a fixed constant $n\geq 0$ and every open subgroup~$\H$ of $\G$, one has 
 \begin{eqnarray*}  - \chi_3(\H) +n  &: = &      -1-\df(\H)+\dim H^3(\H) +n \\
 & \geq  &  [\G:\H]\big(-1-\df(\G) +\dim H^3(\G)\big)  \\
 &  := &  -[\G:\H] \chi_3(\G).  \
 \end{eqnarray*}
 Then $\cd(\G) \leq 3$.
 \end{lemm}

\begin{theo} \label{theo:coho_dimension}
 Let $\K$ be a number field such that 
\begin{itemize}
\item[(i)] $\K$ contains a primitive $p$th root of unity;
\item[(ii)]  $ \O_\K^\times\cap (\O_{\K_\emptyset}^{\times})^p= (\O_\K^{\times})^p$.
\end{itemize}
Then $\dim H^3(\G_\emptyset)>0$. 
Moreover:
\begin{itemize}
 \item[$-$] If $\dim H^3(\G_\emptyset)=1$, then $\G_\emptyset$ is finite or  of cohomological dimension $3$;
 \item[$-$] If $\df(\G_\emptyset)=0$, and if  $\G_\emptyset$ is of cohomological dimension  $3$, then $\G_\emptyset$ is a Poincar\'e duality group.
\end{itemize}
\end{theo}

\begin{proof}
As $\O_\K^\times\cap (\O_{\K_\emptyset}^{\times})^p= (\O_\K^{\times})^p$ and $\delta=1$, one has,  by Theorem \ref{theo:chi}, 
$$\df(\G_\emptyset)=d(\O_\K^\times)-\lambda_{\K_\emptyset/\K}=r_1+r_2-\lambda_{\K_\emptyset/\K}.$$ 
Let $\H$ be an open normal subgroup of $\G_\emptyset$ and set $\F=\K_\emptyset^\H$. 
Proposition \ref{prop:borneinflambda} implies
$\lambda_{\F_\emptyset/\F} \geq \lambda_{\K_\emptyset/\K}[\G_\emptyset : \H]$, 
so Theorem \ref{theo:chi} implies
 $\df(\H) \leq [\G_\emptyset:\H](r_1+r_2- \lambda_{\K_\emptyset/\K})$. Recalling that $\chi_2$ is the Euler characteristic truncated at second cohomology, 
 $$\chi_2(\H)=1+\df(\H) \leq 1 +  [\G_\emptyset:\H](r_1+r_2- \lambda_{\K_\emptyset/\K}),$$
so  $\chi_2(\H)$ cannot  be equal to $[\G_\emptyset:\H]\chi_2(\G_\emptyset)=[\G_\emptyset:\H](1+r_1+r_2-\lambda_{\K_\emptyset/\K})$,   a necessary condition, by Theorem $5.4$ of \cite{Koch}, for $\G_\emptyset$ to be of cohomological dimension $2$. 
Hence $\G_\emptyset$ is not of cohomological $2$ so  $\dim H^3(\G_\emptyset) >0$.

 Now suppose $\G_\emptyset$ is infinite  and $\dim H^3(\G_\emptyset)=1$. 
 By Theorem  \ref{theo:chi} and Proposition \ref{prop:borneinflambda}, one has 
 
 \begin{eqnarray*}
  \big[-1-\df(\H)+\dim H^3(\H)\big]+1&=&\lambda_{\F_\emptyset/\F} -d(\O_\F^\times)+\dim H^3(\H) \\
  & \geq & [\G_\emptyset:\H]\big(\lambda_{\K_\emptyset/\K} -(r_1+r_2)\big)  \\
  &=& [\G_\emptyset:\H] \big(-1-\df(\G_\emptyset)+\dim H^3(\G_\emptyset) \big)
 \end{eqnarray*}
where the last equality follows from Theorem \ref{theo:chi} using that 
$\beta=0$ and  $\dim H^3(\G_\emptyset)=1$.
Now take $n=1$ in Proposition \ref{prop:dim_coho_schmidt} to conclude $\cd(\G_\emptyset)=3$.

\smallskip

Finally, to check that our group is a Poincar\'{e} group, following \cite[Chapter III, \S 7]{NSW},
we verify that $\displaystyle{D_i(\Z/p\Z):= \lim_{\stackrel{\rightarrow}{\UU}} H^i(\UU)^\wedge=0}$ for $i=0,1,2$, where the limit is taken over open subgroups $\UU$ of $\G_\emptyset$ and   the transition maps are  dual to the corestriction. As $\G_\emptyset$ is assumed to be of finite cohomological dimension, then $\G_\emptyset$ is infinite 
and thus $D_0(\Z/p\Z)=0$. 
 Moreover that 
$$D_1(\Z/p\Z) =\lim_{\stackrel{\rightarrow}{\UU}} \UU^{ab}/p=0$$ follows from the proof of the Principal Ideal Theorem: Namely, for a group $\G$ let $\G'$ be its (closed) commutator subgroup and let $\G^{''}$ be the (closed) commutator subgroup of $\G'$. The  key part of the proof of the Principal Ideal Theorem is that the transfer map
$$\mbox{Ver}: \G/\G' \to \G'/\G^{''}$$ is the zero map. As the transfer map is  the dual of the corestriction map, 
$D_1(\Z/p\Z)=0$.

\smallskip

We now show $D_2(\Z/p\Z)=0$. Let $\UU\subset \G_\emptyset$ be open. Taking the $\UU$ -cohomology of the 
short exact sequence of trivial $\UU$-modules 
$$0 \to  \Z/p\Z \to \Q/\Z  \stackrel{p}{\to}\Q/\Z\to  0$$
gives $$H^1(\UU,\Q/\Z) \stackrel{p}{\to} H^1(\UU,\Q/\Z)  \to H^2(\UU,\Z/p\Z)$$
which becomes $(\UU^{ab})^\vee/p \hookrightarrow H^2(\UU,\Z/p\Z)$. 
As $\dim (\UU^{ab})^\vee/p =\dim H^1(\UU)$ and 
 $\df(\UU)=0$, this injection is an isomorphism and
$ H^2(\UU,\Z/p\Z)^\wedge \simeq \UU^{ab}/p \simeq H^1(\UU,\Z/p\Z)^\wedge$. Since the duals of the two corestriction maps are induced by the transfer, $D_1(\Z/p\Z)=0 \implies D_2(\Z/p\Z)=0$. 
\end{proof}

\begin{rema} The first part of Theorem \ref{theo:coho_dimension} extends the following observation that can be deduced from 
the relationship between  Galois cohomology and \'etale cohomology.

We use the  the formalism of étale cohomology as in \cite{Mazur}.
Suppose  $\df(\G_\emptyset)$ is maximal. Then the \'etale version of the Hochschild-Serre spectral sequence with \cite[Theorem 3.4]{Schmidt1} shows that  $H^i(\G_\emptyset) \simeq H^i_{\'et}(\spec \ \O_\K)$ for $i=1,2$. More, 
if $\G_\emptyset$ has cohomological dimension~$2$, then $\G_\emptyset$ is    infinite: by \cite[Lemma 3.7]{Schmidt1}) and the Hochschild-Serre spectral sequence we also get  
 $\{0\}=H^3(\G_\emptyset)\simeq H^3_{\'et}(\spec \ \O_\K) \simeq \mu_{\K,p}$,
where here $\mu_{\K,p}=\langle \zeta_p \rangle\cap \K$ (by \cite[Theorem 3.4]{Schmidt1}). Hence  $\delta$ {\it must} be zero.

\end{rema}

\subsection{On the maximality of $\df(\G_\emptyset)$} \label{subsection_detect}

\subsubsection{Detecting maximality} The strategy of the Hochschild-Serre spectral sequence allows us to prove:

\begin{theo} \label{theo:criteria_nonfree} Suppose there exist  two  linearly disjoint unramified (and nontrivial) $\Z/p$-extensions $\F_1/\K$ and $  \F_2/\K$ such that 
$t_{\G_i}(\E_{\F_i})=0$, $i=1,2$,  where $\G_i=\Gal(\F_i/\K)$. Then $\df(\G_\emptyset)=d(\O_\K^\times)$ is maximal.
Only one such  extension $\F_i/\K$ is sufficient if $\F_i \not \subset\K'(\sqrt[p]{\O_\K^\times})$, which is the case when $\delta=0$.
\end{theo}

\begin{proof}   First note that Lemma  \ref{lemm:t_n} and the fact that $t_{\G_i}(\E_{\F_i})=0$ implies $\lambda=0$. 

\smallskip
If $\delta=0$, Proposition \ref{prop:delta_beta} implies $\beta=0$ so 
$\df(\G_\emptyset)$ is maximal by 
Theorem \ref{theo:chi}.

\smallskip

We now address the $\delta=1$ case.
First suppose $\F_1\nsubset \K(\sqrt[p]{\O_\K^\times})$. Then one can choose 
$d(\O_\K^\times)$ primes $\p$ of $\K$ that split completely in $\F_1$ and whose Frobenii form a basis of 
$\Gal( \K(\sqrt[p]{\O_\K^\times})/\K)$. 
 Since $t_{\G_{1}}(\E_{\F_{1}})=0$, we see that for each $\p \in S_2$ there is a $\Z/p\Z$-extension of 
 $\F_1$, and hence of 
 $\K_\emptyset$, ramified only at (the primes above) $\p$. Each of these elements gives rise to a relation of $\G_\emptyset$. 
  As usual, one gets the rest of the relations 
 "for free" by choosing primes that split completely in $ \K(\sqrt[p]{\O_\K^\times})/\K$ but form a basis of
 $\Gal\left(\K(\sqrt[p]{\V_{\K,\emptyset}})/    \K(\sqrt[p]{\O_\K^\times})\right)$. 
 For such primes $\p$ there is always an abelian extension
 of $\K$ ramified only at $\p$, also giving rise to a relation of $\G_\emptyset$

We study the remaining case, namely when $\F_1, \F_2 \subset \K(\sqrt[p]{\O_\K^\times})$.
Choose a prime $\q_1$ of $\K$ such that its Frobenius 
generates $\Gal(\F_1/\K)$ {\it and} $\q_1$ splits in $\F_2$. Choose $\q_2$ similarly. 
Then, as before, when we allow ramification at $\q_1$ we obtain a ramified extension over $\F_{2,\emptyset}$ and when we allow
ramification at $\q_2$ we obtain a ramified extension over $\F_{1,\emptyset}$.
Set  $S_2=\{\q_1,\q_2\}$ and 
augment $S_2$ to include primes that split completely in $\F_1\F_2$ and whose Frobenii, along with those of $\q_1$ and $\q_2$, form a basis of 
$\Gal( \K(\sqrt[p]{\O_\K^\times})/\K)$. For each of these primes when we allow
ramification at $\p$ we obtain a ramified extension over $\F_{i,\emptyset}$ for $i=1,2$. Each of these primes gives rise
to a relation of $\G_\emptyset$ and along with the "free relations" we get 
 $\df(\G_\emptyset)=d(\O_\K^\times)$.
\end{proof}

\subsubsection{Road to a Conjecture}

 Given $\K_\emptyset/\K$, the condition $\lambda>0$ should be seen as follows: for every  finite Galois extension $\F/\K$ in $\K_\emptyset/\K$ with Galois group $\G$, there exists at least $\lambda$ Minkowski units, {\emph i.e.} $t_\G(\E_\F) \geq \lambda$. 
Regarding relations, we ask:
\begin{Question} When $\K_\emptyset/\K$ is sufficiently large compared to $[\K:\Q]$, in particular when $\K_\emptyset/\K$ is infinite,
 do we have $\df(\G_\emptyset)\geq d(\O_\K^\times)-\beta$ ? In particular, when $\delta=0$ and $\G_\emptyset$ is infinite, do we have $\df(\G_\emptyset)= d(\O_\K^\times)$ ?
\end{Question}

\medskip

In fact, even when $\delta=1$,  we  suspect this to be the case when $d(\G_\emptyset)$ is large compared to $[\K:\Q]$. 
For a number field $\K$, put $\alpha_\K=3+2\sqrt{r_1+r_2+2}$. 
 Note $\alpha_\K$ is {\it small} relative to $d(\O_K^\times)=r_1+r_2-1+\delta$. 

\begin{prop} \label{lemm:avoidance} Assume $\delta=1$ and suppose  $d(\G_\emptyset) \geq  \alpha_\K$.
Then for every cyclic degree $p$-extension $\F/\K$ in $\K(\sqrt[p]{\O_\K^\times})/\K$, there exists an infinite Galois extension $\K^T/\K$ in $\K_\emptyset/\K$ such that $\F \nsubset \K^T$.
\end{prop}
\begin{proof}
Choose a prime ideal $\p\subset \O_\K$ whose Frobenius in $\Gal(\K(\sqrt[p]{\O_\K^\times})/\K)$ generates  $\Gal(\F/\K)$. Put $T=\{\p\}$, and consider $\K^T$ the maximal pro-$p$ extension of $\K$ unramified everywhere
and where $\p$ splits completely. Since  
$$d(\Gal(\K^T/\K))\geq d(\G_\emptyset)-1 \geq 2+2\sqrt{d(\O_\K^\times)+2},$$ the pro-$p$ extension $\K^T/\K$ is infinite (see for example \cite[Th\'eor\`eme 2.1]{Maire_JTNB}). 
And as $\p$ splits completely in $\K^T/\K$, we conclude that $\F \nsubset \K^T$. 
\end{proof}
Proposition \ref{lemm:avoidance} shows the following: Suppose $d(\G_\emptyset) \geq \alpha_\K$. Then  every unit in $\O_\K^\times \cap (\O_{\K_\emptyset}^{\times})^p$ gives an extra relation (following the proof of Theorem \ref{theo:chi}) {\it except} if there is a Minkowski unit in some infinite extension.  In other words: when $\delta=1$, one has $\df(\G_\emptyset) < r_1+r_2$ if and only if there exists a prime ideal $\p \subset \O_\K$ with nontrivial Frobenius in  $\Gal(\K(\sqrt[p]{\O_\K^\times})/\K)$ such that $\K^{\{\p\}}/\K$ has a Minkowski unit.  Here, as $\K^{\{\p\}}/\K$ is infinite, having a Minkowski unit means that for every finite Galois quotient $\L/\K$ of $\K^{\{\p\}}/\K$  of Galois group $\G$, one has $t_\G(\E_\L) >0$. 

\begin{Question} Assume $\delta=1$.
When $d(\G_\emptyset) \geq \alpha_\K$, do we have $\df(\G_\emptyset)=d(\O_\K^\times)$~?
\end{Question}


\section{On the depth of the relations}

It is easier to use software for computations along the Frattini tower of $\K_\emptyset/\K$  than along the Zassenhaus tower. In this section we show the existence of Minkowski units deep in the Frattini tower imply that some of the relations of $\G_\emptyset$ are very deep. This makes it "more likely" that one can prove $\G_\emptyset$ is infinite using the Golod-Shafarevich series. We also prove a converse, namely the existence of very deep relations implies the existence of Minkowski units along the Frattini tower.

\subsection{On the Zassenhaus filtration}
\subsubsection{Basic properties} \label{section_basic_properties} We refer to Lazard \cite[Appendice A3]{Lazard}.
Given a finitely presented prop-$p$ group $\G$, let us take a minimal presentation of $\G$ 
$$1\longrightarrow \R \longrightarrow \F \stackrel{\varphi}{\longrightarrow} \G \longrightarrow 1,$$
where $\F$ is a free pro-$p$ group on $d$ generators; here $d=d(\G)$. Let $\I=\ker(\fq_p\ldbrack \F \rdbrack \rightarrow \fq_p)$ be the augmentation ideal of $\fq_p\ldbrack \F \rdbrack$, and for $n\geq 1$ consider  $\F^n=\{x\in \F, x-1 \in \I^n\}$. The sequence  $(\F^n)$ of open subgroups of $\F$ is the Zassenhaus filtration of $\F$. 

The depth $\omega$ of $x\in \F$ is defined as being $\omega(x)=\max\{n, \ x-1 \in \I^n\}$, with the convention that $\omega(1)=\infty$; the function $\omega$ is a valuation following terminology of Lazard.  Hence $\F^n=\{g\in \F, \omega(g)\geq n\}$.
This allows us to define a depth $\omega_\G$ on $\G$ as follows: $\omega_\G(x)=\max\{\omega(g), g\in \F, \varphi(g)=x\}$.
Put $\G^n=\{x\in \G, \omega_\G(x)\geq n\}$. Observe that $\G^n=\F^n\R/\R\simeq \F^n/\left(\F^n\cap \R\right)$; the  sequence $(\G^n)$  is the Zassenhaus series of $\G$, it corresponds  to the filtration 
arising from the augmentation ideal
$\I_\G$ of $\fq\ldbrack \G \rdbrack$, see \cite[Appendice A3, Theorem 3.5]{Lazard}.
One has  the following property. If $\pi: \G'\twoheadrightarrow  \G$ is surjective, then $\omega_\G$ is the restriction of $\omega_{\G'}$; in other word, $\omega_\G(x)=\max\{\omega_{\G'}(y), y\in \G', \ \pi(y)=x\}$.

Denote by $(\G_n)$ the Frattini filtration of $\G$. 
Recall the well-known relationship between these two filtrations of $\G$:
\begin{lemm} \label{lemm:depth_relation}
One has $\G_{n} \subset \G^{2^{n-1}}$.
\end{lemm}

We say a few words about the reverse inclusions.
Let $\H$ be an open normal subgroup of $\G$.
Since the groups $(\G^n)$ form a basis of neighborhoods   of $1$, let $a(\H)$ be the smallest integer such that $\G^{a(\H)}\subset \H$. We want to give some estimates on $a(\H)$ in some special cases.

\begin{defi}
For a pro-$p$ group $\Gamma$, denote by $\I_{\Gamma}:=\ker(\fq_p\ldbrack \Gamma \rdbrack \rightarrow \fq_p)$, the   augmentation ideal of $\fq_p\ldbrack \Gamma \rdbrack$; and denote by $k(\Gamma)$  the smallest integer such that $\I^{k(\Gamma)}_\Gamma=\{0\}$, where we allow $k(\Gamma)=\infty$.
\end{defi}

\begin{prop} \label{prop_augmentation_ideal}
 Let $1\to \H \to \G \to \G/\H \to \1$ be an exact sequence of pro-$p$ groups. Then $\I_\H=\ker(\fq_p\ldbrack \G \rdbrack \rightarrow \fq_p \ldbrack \G/\H \rdbrack)$.
\end{prop}

\begin{proof}
 See for example \cite[Chapter 7, \S 7.6, Theorem 7.6]{Koch}.
\end{proof}

Following Koch's book \cite[Chapter 7, \S 7.4]{Koch}, we give some estimates  for $a(\H)$. 

\begin{prop} \label{prop_estimate_zassenhaus}
 One has:
 \begin{enumerate}
 \item[$(i)$]  $a(\H) \leq k(\G/\H) \leq |\G/\H|$.
  \item[$(ii)$] If $\Gamma' \lhd \Gamma$ are two finite $p$-groups,  then $k(\Gamma) \leq  k(\Gamma/\Gamma')k(\Gamma') $.
  \item[$(iii)$]  $k(\G/\G_2)=p$. 
 \end{enumerate}
\end{prop}

\begin{proof}
$(i)$ Take $k$ such that $\I_{\G/\H}^k=\{0\}$. Then by Proposition \ref{prop_augmentation_ideal} one has  $\I_\G^k \subset \I_\H$, which implies $\G^k \subset \H$, and then $a(\H) \leq k$.
In particular, $a(\H) \leq k(\G/\H)$. 
For the second part of the inequality, observe that: for every finite $p$-group $\Gamma$, one has  $\I_{\Gamma}^{|\Gamma|}=\{0\}$ (see the proof of Lemma 7.4 of \cite[Chapter 7, \S 7.4]{Koch}), showing that  $k(\Gamma) \leq |\Gamma|$.

$(ii)$ By Proposition \ref{prop_augmentation_ideal},   one has $\I_{\Gamma}^{k(\Gamma/\Gamma')} \subset \I_{\Gamma'}$, and then $\I_{\Gamma}^{k(\Gamma/\Gamma')k(\Gamma')} \subset \I_{\Gamma'}^{k(\Gamma')}=\{0\}$.

$(iii)$ This follows as $\G/\G_{2}$ is  $p$-elementary abelian. 
\end{proof}

For every integer $n\geq 1$, put $a_n:=a(\G_{n+1})$. Observe that $a_1=1$.

  \begin{prop} \label{prop:borne_an}  For $n\geq 2$,
 one has $a_n \leq p^n$. Therefore  $\displaystyle{\G_n \subset \G^{2^{n-1}} \subset \G_{(n-1)\log(2)/\log(p)}}$.
\end{prop}

\begin{proof}
That $a_n \leq p^n$  follows from Proposition~\ref{prop_estimate_zassenhaus} and the fact that $\G_n/\G_{n+1}$ is elementary $p$-abelian. The second part follows from the first.
\end{proof}

\subsubsection{The Golod-Shafarevich polynomial}\label{sec:GS} Consider a minimal presentation of a finitely generated pro-$p$ group $\G$:
$$1\longrightarrow \R \longrightarrow \F \stackrel{\varphi}{\longrightarrow} \G \longrightarrow 1.$$
Suppose that  $\R/\R^p[\R,\R]$ is generated as an $\fq_p\ldbrack \F \rdbrack$-module by the family $\ff=(\rho_i)$ of elements $\rho_i \in \F$. For $k\geq 2$, put $r_k=|\{\rho_i, \omega(\rho_i)= k\}|$; here we assume the $r_i$'s finite.

\begin{defi}
The series $P(t)= 1-d t +\sum_{k\geq 2} r_k t^k$ is a Golod-Shafarevich series associated to the presentation $\ff$ of $\G$.
\end{defi}

 The theorem of  Golod-Shafarevich implies the following: if  for some $t_0 \in (0,1)$ one has $P(t_0)=0$, then $\G$ is infinite (see \cite{Vinberg} or \cite{Andozski}). 
Observe that when  no information on the depth of the $\rho_i$ is available, then one may take  $1-dt+rt^2$ as Golod-Shafarevich series for $\G$, where  $r=d(H^2(\G))$.

\begin{rema}
When $\G=\G_\emptyset$,  the $p$-rank of $\G_{n}$ corresponds to the $p$-rank of the class group of $\K_n$, where $\K_n=\K_\emptyset^{\G_{n}}$. Hence by Class Field Theory and with the help of a software package, in a certain sense it is easier to test if an element of 
$\G$ is in $\G_{n}$ than if it is in $\G^n$.
\end{rema}

\subsection{Minkowski units and the Golod-Shafarevich polynomial of $\G_\emptyset$}

\subsubsection{The principle} \label{section_principle}

Let $S$ be a finite saturated set of tame places of~$\K$ as in Lemma \ref{lemma1}, {\it i.e.} such that $H^1(\G_\emptyset) \simeq H^1(\G_S)$ and $|S| = d(\V_{\K,\emptyset})$. 
 Put $d=d(\G_\emptyset)$. Let $\F$ be the free pro-$p$ group on $d$ generators  $x_1,\cdots, x_d$. Consider now  the minimal presentations of $\G_\emptyset$ and $\G_S$ induced by $\F$, and  the following diagram
$$
\xymatrix{ 1  \ar[r] & \R_S \ar[r]^{i'} \ar[d] & \F \ar[r]^{\varphi_S} \ar@{->}[d]^{=} & \G_S \ar[r] \ar@{->>}[d] & 1 \\
 1  \ar[r] & \R_\emptyset \ar[r]^i & \F \ar[r]^\varphi & \G_\emptyset \ar[r] & 1 \\
}$$
 Put $\H_S=\ker(\G_S \to \G_\emptyset)$; the pro-$p$ group $\H_S$ is the normal subgroup of $\G_S$ generated by the tame inertia groups $\langle \tau_\p \rangle_{\p \in S}$. Hence, this  diagram induces the following exact sequence
$$1 \to \R_S \to \R_\emptyset \stackrel{\psi'}{\to} \H_S \to 1,$$
where $\psi'=\varphi_S\circ i$.
 The Hochschild-Serre spectral sequences induce the following isomorphisms 
 $$
 \xymatrix{& H^2(\G_\emptyset)^\wedge \ar[dl]_{\simeq} \ar[dr]^{\simeq}& \\
 \R_\emptyset/\R_\emptyset^p[\R_\emptyset,\F]  \ar@{->}[rr]_{\simeq}^\psi &&   \H_S/\H_S^p[\G_\emptyset,\H_S]}$$
 
 where  $\psi$ is induced by $\psi'$. 
 Using $\psi$ we will study 
 the depth of the relations of $\G$: indeed,  $\R_\emptyset/\R_\emptyset^p[\R_\emptyset,\F]$ and $\H_S/\H_S^p[\G_\emptyset,\H_S]$ inherit the Zassenhaus valuation from $\R_\emptyset$ and $\H_S$, and then on the Zassenhaus valuation on~$\F$.
 Therefore an element of depth $k$ in $\R_\emptyset/\R_\emptyset^p[\R_\emptyset,\F$ corresponds to one element     
 of depth $k$ in $\H_S/\H_S^p[\G_\emptyset, \H_S]$.

\subsubsection{Minkowski element}
Here we extend the notion of Minkowski unit to the notion of Minkowski element. 
Set $\VV_\K=\V_{\K,\emptyset}/(\K^{\times})^p$.

\begin{defi}
Let $\L/\K$ be a Galois extension with Galois group $\G$. We denote by $\lambda_{\L/\K}':=t_\G(\VV_{\L})$ the $\fq_p[\G]$-rank of $\VV_{\L}$. One says that $\L/\K$ has a {\em Minkowski element} if $\lambda_{\L/\K}'\geq 1$. 
\end{defi}

\begin{lemm}
 One has $\lambda_{\L/\K}' \geq \lambda_{\L/\K}$,so the existence  of a Minkowski unit implies that of a Minkowski element.
\end{lemm}
\begin{proof} This follows immediately from  the exact sequence 
\begin{eqnarray} \label{classical} 1 \longrightarrow \E_\L \longrightarrow \VV_{\L} \longrightarrow \Cl[p] \longrightarrow 1.
 \end{eqnarray}
\end{proof}

When $\L/\K$ is a subextention of  $\K_\emptyset/\K$ one may give an upper bound for $\lambda_{\L/\K}'$:

\begin{prop}\label{prop:bound_lambda'}
 Let $\L/\K$ be a nontrivial  finite Galois extension in $\K_\emptyset/\K$. Then $\lambda_{\L/\K}'\leq d-1+r_1+r_2$.
 Moreover, if $\K_\emptyset/\K$ is infinite then  there exist infinitely many Galois extensions $\L/\K$ in $\K_\emptyset/\K$ such that $\lambda_{\L/\K}'< d-1+r_1+r_2$.
\end{prop}

\begin{proof} Set $\H=\Gal(\K_\emptyset/\L)$.
By Schreier's inequality (see \cite{RZ}, Corollary 3.6.3), $d(\Cl_\L)= d(\H)\leq |\G/\H|(d-1)+1$. Hence by (\ref{classical}), we get
$$d(\VV_{\L} )\leq |\G/\H|(d-1+r_1+r_2)+\delta,$$
showing that $\lambda_{\L/\K}'\leq d-1+r_1+r_2$.

Suppose now that $\G_\emptyset$ is infinite and, except for finitely many Galois extensions $\L/\K$ in $\K_\emptyset/\K$, one has  $\lambda_{\L/\K}'= d-1+r_1+r_2$.
Then $d(\VV_\L )\geq  |\G/\H|(d-1+r_1+r_2)$ and $$d(\Cl_\L) \geq 1-\delta+|\G/\H|(d-1)\geq |\G/\H|(d-1),$$
implying $$-\chi_1(\H)+1 \geq-|\G/\H|\chi_1(\G_\emptyset).$$
By \cite[Proposition 1]{Schmidt-BLMS}, the Galois group $\G_\emptyset$ must be free pro-$p$, which is impossible.
\end{proof}

The converse of Lemma \ref{lemma:minkowski-unit-1} follows easily from the Cheobatarev density theorem:

\begin{prop} Let $\L/\K$ be a  finite $p$-extension of Galois group $\G$. 
\begin{enumerate}
\item[$(i)$] If $t_\G(\E_{\L,\emptyset})\geq k$, then there exist  infinitely many sets $S=\{\p_1, \cdots, \p_k\}$ of tame primes of~$\K$ such that $\#\G_{\L,S}^{ab}=\#\G_{\L,\emptyset}^{ab}$. 
\item[$(ii)$] If  $t_\G(\VV_{\L})\geq k$, then there exist infinitely many sets $S=\{\p_1, \cdots, \p_k\}$ of tame primes of~$\K$ such that $d(\G_{\L,S})=d(\G_{\L,\emptyset})$.
\end{enumerate}
\end{prop}

From the computational view point,  we will now consider the sequence of fields $(\K_n)$ in $\K_\emptyset/\K$ induced by the Frattini filtration $(\G_{n})$: in other word, $\K_n=\K_\emptyset^{\G_{n}}$.  Put $\H_n=\Gal(\K_n/\K)$, and denote by $\lambda_n':=\lambda_{\K_n}'$ the $\fq_p[\H_n]$-free rank of $\VV_{\K_n}$.

Put $d:=d(\G_\emptyset)$, and $r_{\textrm{max}}:=d+d(\O_\K^\times)$.

\begin{theo} \label{main-theorem2} Take $n\geq 2$.  Then   $\G_\emptyset$  can be generated by $d(\G_\emptyset)$ generators and   $r_{\textrm{max}}$ relations $\{\rho_1,\cdots, \rho_{r_{\textrm{max}}}\}$ such that at least $\lambda_n'$  relations are of depth greater than $2^{n}$.
\end{theo}

\begin{proof}
 We are assuming that the $\fq_p[\H_n]$-module $\VV_{\K_n}$ is isomorphic to $\fq_p[\H_n]^{\lambda_n'} \oplus N$ where $N$ is torsion. Using Chebotarev's theorem, choose a set $S'=\{\p_1,\cdots,\p_{\lambda_n'}\}$ of primes of $\K$ such that
 \begin{itemize} 
 \item Each $\p_i$ splits completely from $\K$ to $\K_n$,
 \item The Frobenius at a prime $\P_{ij}$ of $\K_n$ above $\p_i$ in $\Gal(\K'_n(\sqrt[p]{\VV_{\K_n}})/\K'_n)$ lies in the $i$th copy of $\fq_p[\H_n]  \subset\VV_{\K_n}$ and generates that copy of $\fq_p[\H_n]$ under the action of $\H_n$.
 \end{itemize}
 We claim $d(\G_{\K_n,S'})= d(\G_{\K_n,\emptyset})$. 
Indeed, there are $|\H_n|$ primes $\P_{ij}$ of $\K_n$ above $\p_i$ and they have independent Frobenii in   $\Gal(\K'_n(\sqrt[p]{\VV_{\K_n}})/\K'_n)$ by choice, even as we take the union over  $i$ from $1$ to $\lambda_n'$. 
 Gras-Munnier (Theorem \ref{theo:Gras-Munnier}) gives the equality. In fact, it gives more: $d(\G_{\K_m,S'})= d(\G_{\K_m,\emptyset})$ for all $m<n$. If this were false for $m_0<n$, there would exist a $\Z/p\Z$-extension of $\L/\K_{m_0}$ ramified at primes (above those) of $S'$. Thus $\L\K_n/\K_n$ would be  a $\Z/p\Z$-extension ramified only at primes (above those) of $S'$ contradicting the result for $n$. We have shown that the $p$-Frattini towers of $\G_{S'}$ and $\G_\emptyset$ agree at the first $n$ levels. Thus the generators $\tau_{\p_i}$ of the tame inertia groups all have depth $2^n$ in $\G_{S'}$.
  
 We have $$0 \to \sha^2_{S'} \to H^2(\G_{S'}) \stackrel{res}{\to} \oplus_{\p_i\in S'} H^2(\G_{\p_i})$$
 and $\dim \sha^2_{S'} \leq \dim \CyB_{S'} = r_{\textrm{max}} -\lambda_n'$. We can say nothing about the depth of the relations
 coming from $\sha^2_{S'}$, so we assume they have minimal depth two. The local relations are of the form $[\sigma_{\p_i},\tau_{\p_i}]\tau_{\p_i}^{N(\p_i)-1}$and are easily seen to have depth at least $2^n+1$. As 
 $$G_{S'} / \langle\tau_{\p_1},\cdots,\tau_{\p_{\lambda_n'}}\rangle \simeq\G_\emptyset$$
 and taking this quotient trivializes the local relations, the theorem follows.
 \end{proof}

\begin{rema}
 If $\K_n$ corresponds to the Zassenhaus filtration, we gets that the $\lambda_n'$ relations are of depth at least $2n$. 
\end{rema}

\begin{coro}\label{coro:weighted_relations} 
 Assuming the hypothesis of Theorem \ref{main-theorem2}, we may take $1-dt+(r_{\textrm{max}}-\lambda_n')t^2+\lambda_n' t^{2^n}$
as a Golod-Shafarevich polynomial for $\G_\emptyset$. 
\end{coro}

\begin{exem}
Let us return to the field $\K=\Q(\sqrt{5\cdot 13 \cdot 17 \cdot 19})$. Take $\H=\K_2$, $\G=\Gal(\H/\K)$. As seen earlier, $t_\G(\VV_H)\geq 1$. Indeed, the existence of a Minkoswki element follows from that of a Minkowski unit. Here a Golod-Shafarevich polynomial of $\G_\emptyset$ can be taken as: $1-3t+4t^2+t^4$.
\end{exem}

\begin{coro} 
 If $[\K:\Q]$ is large compared to $d$, then $\G_\emptyset$ has some deep relations. More precisely,
 one may take  $1-dt+(r_{\textrm{max}}-a)t^2 + a t^{4}$ as Golod-Shafarevich polynomial for $\G_\emptyset$, for some integer $a>0$.
\end{coro}

\begin{proof} 
 By Theorem \ref{theo:ozaki} applied to $\G=\G_\emptyset/\G_\emptyset^2[\G_\emptyset,\G_\emptyset]=\G_\emptyset/\G_\emptyset^2$, one observes that $t_\G(\E_{\K_2})\geq a$, if $[\K:\Q]$ is large as compared to a certain quantity depending on $\G$.
Now apply Corollary  \ref{coro:weighted_relations} since $\lambda_2'\geq t_\G(\E_{\K_2})$. 
\end{proof}

\subsubsection{The converse}  Theorem \ref{main-theorem2} shows that the presence of Minkowski elements in the tower implies the existence of very deep relations in $\G_\emptyset$. Here we show the converse, that the existence of very deep relations implies the presence of Minkowski elements.

For $n\geq 1$, recall that $\F$ is a free pro-$p$ group on $d$ generators, $\F^{m}$ and $\F_m$ are the Zassenhaus and Frattini filtrations, and $a_n$ is the smallest integer such that $\F^{a_n} \subset \F_{n+1}$. 
Recall from Lemma \ref{lemm:depth_relation} that $\F_n \subset \F^{2^{n-1}}$.
See Section \ref{section_basic_properties}. 
Put $\H_n=\G_\emptyset/\G_{n}$ and $\K_n=\K_\emptyset^{\G_{n}}$.

\begin{theo} \label{theorem_reverse_2} Suppose that all the relations of $\G_\emptyset$ are of depth at least $a_n$. Then 
\begin{itemize}
 \item[$(i)$] if $\zeta_p\in \K$, $\lambda_{\K_{n}/\K}' \geq r_1+r_2$;
 \item[$(ii)$] if $\zeta_p\notin \K$, $\lambda_{\K_{n}/\K}' = r_1+r_2-1+d$.
\end{itemize}
\end{theo}

 \begin{proof} 
 Since all the relations of $\G_\emptyset$ have depth $a_n$, we see that 
 $\G_\emptyset/\G^{a_n}_\emptyset \simeq \F/\F^{a_n}$ has maximal Zassenhaus filtration for the first $a_n$ steps. 
Thus for  {\it any}  set $S$ satisfying $d(\G_S)=d(\G_\emptyset)$ we have 
$$\F/\F^{a_n}\simeq \G_\emptyset/\G_\emptyset^{a_n}\simeq \G_S/\G_S^{a_n} $$
and since $\F^{a_n} \subset \F_{n+1}$, we also have
$$\F/\F_{n+1}\simeq \G_\emptyset/\G_{\emptyset,n+1}\simeq \G_S/\G_{S,n+1} $$
so all relations of $\G_\emptyset$ have depth at least $n+1$ in the Frattini filtration.

 \smallskip
 
 We first address the case   $\delta=1$.
 Consider the $p$-elementary abelian extensions $\K(\sqrt[p]{\V_{\K,\emptyset}})/\K$ and $\K_2/\K$, the latter being the maximal unramified $p$-elementary abelian extension of $\K$. By Kummer theory each is formed by adjoining to $\K$ the $p$th roots of elements $\alpha \in \K$. Since $\K_2/\K$ is everywhere unramified,  $(\alpha)$ is the $p$th power of an ideal, that is $\alpha \in \V_{\K,\emptyset}$ so  $\K(\sqrt[p]{\V_{\K,\emptyset}}) \supset \K_2$ and 
 $d( \Gal ( \K(\sqrt[p]{\V_{\K,\emptyset}}) / \K_2 ))=r_1+r_2$. 
 Note $\K_n \cap \K(\sqrt[p]{\V_{\K,\emptyset}}) =\K_2$ as the intersection is both unramified over $\K$ and $p$-elementary abelian over $\K$.
 Let $S:=\{\p_1,\cdots,\p_{r_1+r_2}\}$ consist of primes that split completely from $\K$ to $\K_2$ to $\K_n$ and whose Frobenii form a basis of $\Gal (\K(\sqrt[p]{\V_{\K,\emptyset}})/\K_2) $.  
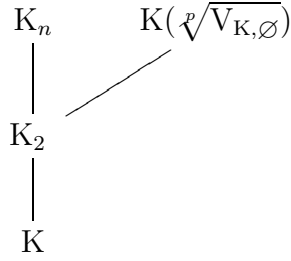
 $$\displaystyle
 \xymatrix{&  \K_n  \ar@{-}[d] & \K(\sqrt[p]{\V_{\K,\emptyset}})  \\
& \K_2\ \ar@{-}[ur] & \\
& \K \ar@{-}[u]&  }
$$
 By the above discussion
 \begin{eqnarray}\label{FrattMink}
\F/\F_{n+1}\simeq \G_\emptyset/\G_{\emptyset,n+1}\simeq \G_S/\G_{S,n+1}.
 \end{eqnarray}
This will imply $\lambda'_{\K_n/\K}\geq r_1+r_2$. Indeed, 
above each $\p_i$ there are $[\K_n:\K]$ primes $\P_{ij}$ in $\K_n$ upon which $\Gal (\K_n/\K)$ acts transitively. 
If for some $i$ the Frobenii of the $\P_{ij}$ did not generate a distinct copy of $\fq_p[\G_n]$   in 
$\Gal(\K_n(\sqrt[p]{\V_{\K_n,\emptyset}})/\K_n)$, then 
there would be a dependence relation among them and by Gras-Munnier we would have $d(\G_{\K_n,S}) > d(\G_{\K_n,\emptyset})$, contradicting (\ref{FrattMink}). Thus $\lambda'_n \geq r_1+r_2$ completing the proof in the $\delta=1$ case.

 \smallskip
 
We now consider the case  $\delta=0$. 
As usual, the key fact is  that $\K'_n\cap \K'(\sqrt[p]{\V_{\K,\emptyset}})=\K'$ (following the proof of Proposition~\ref{prop:delta_beta}) 
so $d(\Gal(\K'_n (\sqrt[p]{\V_{\K,\emptyset}}) /\K'_n))=r_1+r_2-1+d$. 

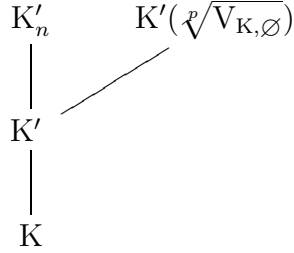
$$\displaystyle
 \xymatrix{ \K'_n  \ar@{-}[d] & \K'(\sqrt[p]{\V_{\K,\emptyset}})  \\
 \K'\ \ar@{-}[ur] & \\
 \K \ar@{-}[u]& }
$$
We choose $S:=\{\p_1,\cdots,\p_{r_1+r_2-1+d}\}$ to consist of primes of $\K$ that split completely from $\K$ to $\K'$ to $\K'_n$ and whose Frobenii form a basis of $\Gal (\K(\sqrt[p]{\V_{\K,\emptyset}})/\K') $.  
We complete the proof exactly as in the $\delta=1$ case.
\end{proof}

 \begin{coro}If all the relations of $\G_\emptyset$ are of depth at least $p^2$ then $\K_2$ has a Minkowski element. 
 \end{coro}
 
 \begin{proof}
This follows immediately from Proposition \ref{prop:borne_an} and Theorem \ref{theorem_reverse_2}.
 \end{proof}

\subsection{The alternative}

There is another way by which we can obtain Theorem \ref{theo:chi} in the context of Golod-Shafarevich series $P(t)$. Indeed, such a series for a pro-$p$ group $\G$ approximates the Hilbert $H_\G$ series of  the Zassenhaus filtration of $\G$.  In particular the Golod-Shafarevich Theorem is a consequence of this inequality:  if there is some $t_0\in]0,1[ $ such that $P(t_0)<0$ then necessarily $H_\G(t_0)$ diverges, implying the infiniteness of $\G$.

Retain the notations of Section \S \ref{section:alongthetower}, and fix $n \gg 0$. Apply Corollary \ref{coro:weighted_relations} to $\K_n/\K$ 
by taking $1-dt+(r_{\textrm{max}}-\lambda)t^2+\lambda t^{2^n}$
as a Golod-Shafarevich polynomial for $\G_\emptyset$. 
Now, as~$n$ can be arbitrarly large, we see that $1-dt+(r_{\textrm{max}}-\lambda)t^2$  is a Golod-Shafarevich polynomial for $\G_\emptyset$. 
\smallskip

Of course, the question of determining $\lambda$ when it is nonzero seems a hard problem, except in the case where at  the beginning of the tower, we see $\lambda=0$.
Here is an explicit alternative.

\medskip

\begin{coro} \label{theo:depth} Let $n \in \Z_{>1}$. One has:
\begin{enumerate}
\item[$(i)$] if $t_{\H_n}(\E_{\K_n})=0$ and $\beta=0$, then $\df(\G_\emptyset)=r_1+r_2-1+\delta$;
\item[$(ii)$] if $t_{\H_n}(\E_{\K_n})=\lambda_n >0$, then one may take $1-dt+(r_{\textrm{max}}-\lambda_n)t^2+\lambda_n t^{2n}$ as a Golod-Shafarevich polynomial for $\G_\emptyset$.
\end{enumerate}
\end{coro} 

\begin{rema}
 The condition $\beta=0$ can be relaxed as noted in Theorem \ref{theo:criteria_nonfree}.
\end{rema}


\section{The case of imaginary quadratic fields}

In this section, we take $p=2$ and let $\K:=\Q(\sqrt{D})$  be an imaginary quadratic  field of discriminant $D<-7$. Since the unit rank of $\K$ is $1$, we have $\df(\G_\emptyset)\in \{0,1\}$.
In this simplest of all non-trivial situations, we will discuss the deficiency of $\G_\emptyset$ and explore the extent to which we can detect relations using the machinery and notation set up in Section \ref{section:2.2}.

\subsection{The frame}
Let $d=d(\Cl_\K)$ be the $2$-rank of the class group of $\K:=\Q(\sqrt{D})$. 
By Gauss's genus theory, we know that $D$ admits a unique (up to reordering) factorization into $d+1$ integers, each of which is a ``prime fundamental discriminant'' --  meaning it is the discriminant of a quadratic field in which a single prime ramifies. For an odd prime $q$, we define $q^*:= (-1)^{(q-1)/2}q$. The prime discriminants are then $q^*$ as $q$ ranges over all odd primes, as well as $-4$ and $\pm 8$. We write $D=q_1^*\cdots q_{d+1}^*$, with the convention that if $D$ is even, then $q_{d+1}^* \in \{-4,-8,8\}$.
 
 Put  $q_0^*=-1$ and for each $i$ in the range $0\leq i \leq d$, put
 $$ \K_i:=\K(\sqrt{q_0^*},\cdots, \sqrt{q_{i-1}^*},\sqrt{q_{i+1}^*},\cdots, \sqrt{q_{d}'}),$$
 where 
$$
q'_d=\left\{\begin{array}{ll}
             q_d^* & {\rm \ if \ } D {\rm \ is \ odd} \\
              q_d^* & {\rm \ if \ } q_{d+1}^*=\pm 8 \\
              2 & {\rm \ if \ } q_{d+1}^*= -4.
            \end{array}
\right.
$$ 
Also define   $\L':=\K(\sqrt{q_0^*},\sqrt{q_1^*},\cdots, \sqrt{q_{d-1}^*},\sqrt{q_d'})$. A direct computation shows that the number field $\L'$ is the governing field $\K(\sqrt{\V_\emptyset})$ (see Section \ref{section:2.2}).
Choose prime numbers $p_0,\cdots, p_{d}$ that split in $\K$ and  such that for each $i$ in the range $0\leq i \leq d$, the Frobenii of the  $p_j, j\neq i$  in $\L'/\Q$ generate the Galois group of the quadratic extension 
$\L'/\K_i.$
Fix  a prime $\p_i|p_i$ of $\K$ and put  $S_2=\{\p_0\}$, $S_1=\{\p_1,\cdots,\p_d\}$, and $S=S_1 \cup S_2$.
Observe that the primes $p_1,\cdots, p_d$  all are congruent to $1$ mod $4$ and that $p_0 \equiv 3$ mod $4$.

\smallskip

As the $2$-part of the class group of $\K$ has $d$ generators, Lemma \ref{lemma:deploye} shows the existence of  $d$ independent quadratic extensions $\F_i$ above $\K_\emptyset$, totally ramified at $\p_i$, $i=1,\cdots, d$, so $d(\X_S) \geq d$. This puts us in the situation where the extra relations are detectable by the set $S_2$. Now, by studying the Galois module structure of units in imaginary biquadratic number fields, we can specify conditions under which $\df(\G_\emptyset)=1$; see Theorem \ref{main-theorem} below.

\begin{lemm} \label{lemm:minkowski-unit}
Let $\K_0/\Q$ be a real quadratic field; $\G_0=\Gal(\K_0/\Q)$. Then $\E_{\K_0}$ is $\fq_2[\G_0]$-free if and only if, the norm of the fundamental unit $\varepsilon$ is $-1$. 
More precisely, as an $\fq_2[\G_0]$-module,
$\E_\K \simeq \left\{ \begin{array}{ll}  \fq_2 \oplus \fq_2 & N(\varepsilon)=1 \\ \fq_2[\G_0] & N(\varepsilon)=-1 \end{array}\right.$.
\end{lemm}

\begin{proof}
If the norm of $\varepsilon$ is $+1$, then modulo $(\O_K^\times)^2$, we get $\varepsilon^\sigma \equiv \varepsilon$.
If the norm of $\varepsilon$ is $-1$, then $\E_\K$ is generated by $\varepsilon (\O_\K^\times)^2$ as $\G$-module, and $\langle \varepsilon (\O_\K^\times)^2 \rangle$ is $\fq_2[\G_0]$-free.
\end{proof}

Recall  this well-known result:

\begin{lemm} \label{lemma:fund-unit}
Let $\F/\Q$ be an imaginary biquadratic field. Let $\K_0$ be the real quadratic subfield, and let $\varepsilon$ be the fundamental unit of $\K_0$. Then, $|\O_\F^\times/\langle \mu_\F, \varepsilon \rangle|=1$ or $2$.  In particular, if $\F/\K_0$ is ramified at some odd prime, then  $\O_\F^\times=\langle \mu_\F, \varepsilon \rangle$.
\end{lemm}

\subsection{Main result}

We can now prove:

\begin{theo} \label{main-theorem} Let $\K$ be an imaginary quadratic  field of discriminant $D$.
Assume that we can write $D=D_1D_2$, where $D_1 >0$ and $D_2$ are  fundamental discriminants, such that: 
\begin{enumerate}
\item[$(i)$] the norm of the fundamental unit of $\Q(\sqrt{D_1})$ is $+1$,
\item[$(ii)$]   some odd prime number divides $D_2$. 
\end{enumerate} Then $\df(\G_\emptyset)=1$, and the extra relation is detected by  the quadratic extension  $\K(\sqrt{D_1})/\K$.
\end{theo}

\begin{proof} 
Put $\F:=\K(\sqrt{D_1})$. As $D_1$ and $D_2$ are fundamental discriminants, then $\F/\K$ is unramified.
By  assumption $(ii)$ and Lemma \ref{lemma:fund-unit}, $\O_\F^\times = \langle \varepsilon , -1 \rangle$, where $\varepsilon$ is the fundamental unit of $\Q(\sqrt{D_1})$. By assumption $(i)$ and  Lemma \ref{lemm:minkowski-unit},
$\E_\F$ is not  $\fq_2[\G]$-free, where $\G=\Gal(\F/\K)$: in other words $t_{\G}(\E_{\F})=0$.
Now we can conclude by Theorem \ref{theo:criteria_nonfree} (here $\sqrt{-1} \notin \F$).
\end{proof}

\begin{rema} To elaborate further, observe that $p_0$ splits in $\F/\K$.
Indeed, by the choice of $p_0$ we have, for $i=1, \cdots, d-1$,   $\left(\frac{q_i^*}{p_0}\right)=\left(\frac{q_d'}{p_0}\right)=1$.
Let us study two cases.

$(a)$ Suppose first that $q_d'=q_d^*$.
 Then by recalling that  $\left(\frac{D}{p_0}\right)=1$, one also gets 
$\left(\frac{q_{d+1}^*}{p_0}\right)=1$, and then  $\left(\frac{D_1}{p_0}\right)=1$ (here $D_1$ is the product of some of the  $q_i^*$). 

$(b)$ Suppose now that $q_d'=2$. Since $p_0 \equiv 3$ mod $4$ and  $D=q_1^*\cdots q_{d+1}^*$,  we have $\left(\frac{q_{d}^*}{p_0}\right)=-1$. By assumption, there exists an odd prime $p$ that divides $D_2$. We may choose $p=q_{d}$ (before fixing $p_0$). Then, $D_1$ is the product of various $q_i^*$, for $i=1,\cdots, d-1$ so $\left(\frac{D_1}{p_0}\right)=1$.

\medskip
   
As $p_0$ splits completely in $\F/\K$, we see $\prod_{\P|\p_0} \U_\P/\U_\P^2$ is $\fq_2[\G]$-free of rank $1$.
But as $t_\G(\E_\F)=0$, the subgroup $I_{\p_0}$ of $\RCG{\F}{\p_0}$ generated by the ramification at $\p_0$ is not trivial.   Put $I:=I_{\p_0}/I_{\p_0}^2$. By  Nakayama's lemma, the coinvariants $I_\G$ are also not trivial, hence there exists at least one quadratic extension $\F_1/\F_\emptyset$, Galois over $\K$, totally ramified at some $\P|\p_0$, such that $\G$ acts trivially on $\Gal(\F_1/\F_\emptyset)$. The compositum 
$\F_1\K_\emptyset/\K_\emptyset$ is ramified at $\p_0$ and produces a $(d+1)$st  relation. This is the formalism of example \ref{field:5460}.
\end{rema}

\begin{coro}\label{coro_3primes}
Let $\K$ be an imaginary quadratic field of discriminant $D$. Suppose $D$ is divisible by at least two odd primes $p_1,p_2$ such that $p_1 \equiv p_2 \equiv 3 $ mod  $4$. Then
$\df(\G_\emptyset)=1$.
\end{coro}

\begin{proof}
 If there is another odd prime $q$ that divides $D$, take $D_1=p_1p_2$.
 
 If $\K=\Q(\sqrt{-p_1p_2})$ (resp. $\Q(\sqrt{-2p_1p_2})$), take $D_1=4p_1$ (resp. $D_1=8p_1$).
\end{proof}

\begin{exem}[Martinet \cite{Martinet}] 
Take $\K=\Q(\sqrt{-21})$. Then, by Odlyzko bounds the $2$-tower $\K_\emptyset/\K$ is finite, and   it is not hard to see $\G_\emptyset \simeq (\Z/2\Z)^2$, and $\df(\G_\emptyset)=1$.
\end{exem}

\begin{exem}[See Example \ref{field:5460}]
Take $\K=\Q(\sqrt{-5460})$, $D_1=21$ and $D_2=-260$.  We then get an extra relation coming from the extension $\K(\sqrt{21})/\K$, and $\df(\G_\emptyset)=1$.
\end{exem}

\begin{coro} \label{coro:relation-maxi}
Suppose $k\geq 2$, and $p_1, \ldots, p_k$ are $k$ distinct odd primes, exactly one of which, say $p_1$, is  $\equiv 3 \bmod{4}$. 
For the imaginary quadratic field $\K=\Q(\sqrt{-2  p_1\cdots p_k})$  with discriminant $D=-8 p_1 \cdots p_k$, we have: $\df(\G_\emptyset)=1$.
\end{coro}

\begin{proof}
Take $D_1=8p_1$.
\end{proof}

\begin{exem} Take $\K=\Q(\sqrt{-p_1 p_2})$, with primes $p_1, p_2$ such that $p_1\equiv 1$ mod $4$ and $p_2\equiv 3$ mod $4$. Here  the hypotheses of  Theorem \ref{main-theorem} do not apply and $r=d=1$ so $\df(\G_\emptyset)=0$. 
\end{exem}

\begin{exem} The hypotheses of Theorem \ref{main-theorem} do not apply for $\K=\Q(\sqrt{-130})$. As noted by Martinet \cite{Martinet}, in that case, 
$\G_\emptyset$ is the quaternion group so 
$r=d=2$. 
\end{exem}

\begin{exem}
Take $\K=\Q(\sqrt{-5\cdot 13 \cdot 41})$. Here $r=d+1=3$; indeed the norm of the fundamental unit  of $\Q(\sqrt{5\cdot 41})$ is $+1$. 
\end{exem}

\subsection{$\df(\G_\emptyset)$ is maximal almost all the time}

We easily deduce from Theorem \ref{main-theorem} that the presence of a Minkowski unit in a quadratic unramified extension $\F/\K$ is rare, with the consequence that, generically, the deficiency of  $\G_\emptyset$ is maximal.
Let us say more precisely what we mean by the term ``generically'' here.  Denote by $\ff$ the set of imaginary quadratic fields. For $X\geq 2$, put $$\ff(X)=\{ \K \in \ff, \ |\disc(\K)|\leq X\},$$
and $$\ff_0(X)=\{\K \in \ff(X), \ \df(\G_\emptyset)=0\}.$$

\begin{theo} \label{theo:density_quadratic} There is an absolute constant $C>0$ such that for all $X$, 
$$\frac{\#\ff_0(X)}{\#\ff(X)} \leq C\frac{\log \log X}{\sqrt{\log X}}\cdot$$ In particular, when ordered by absolute value of the discriminant, the proportion of imaginary quadratic fields for which $\df(\G_\emptyset)=0$, tends to zero when $X \rightarrow \infty$. 
\end{theo}

\begin{proof} 
 We use the number theory analytic tools of  \cite[Theorem 4.6]{Maire-cupproducts} due to Fouvry. Let $\K$ be an imaginary quadratic field.
For $X \geq 2$, put $$B(X)=\{\K \in \ff(X), \exists \ 2 \ {\rm distincts \ odd \ primes} \  p\equiv q\equiv 3 \ {\rm mod} \  4, pq \mid \disc(K)\}.$$
By Corollary \ref{coro_3primes}, for every $\K \in B(X)$ one has $\df(\G_\emptyset)=1$. Hence $\ff_0(X) $ is  in the complementary $C(X)$ of $B(X)$.

Note by $A_{i}(X)$ the set of square-free integers $n\leq X$ having exactly $i$ prime factors $\equiv 3$ mod $4$,  put $A(X)=A_{0}(X)\cup A_{1}(X)$.
Clearly, $|C(X)|=O(|A(X)|)$.

In the proof of Theorem 4.6 of \cite{Maire-cupproducts}, it is shown that uniformly in $X\geq 2$, one has 
$|A_{0}(X)|=O\Big(X/\sqrt{\log X}\Big)$ and $\displaystyle{|A_{1}(X)|=O\Big(X \frac{\log \log X}{\sqrt{\log X}}\Big)}$.
Thus $\displaystyle{|C(X)|=O\Big(X \frac{\log \log X}{\sqrt{\log X}}\Big)} $.
We conclude by noting that $\displaystyle{|\ff(X)| = \frac{3}{\pi^2} X +O(\sqrt{\log X})}$ (see for example \cite[\S 4]{Fouvry-Klueners}).
\end{proof}




\end{document}